\documentclass{amsart}
\usepackage{amsmath,amscd,amssymb,amsthm,amsbsy,amscd,stmaryrd}
\usepackage[dvips]{graphicx,color}
\usepackage{epsfig}
\usepackage{mathrsfs}
\usepackage{mathpazo}
\input{xy}

\newtheorem{theorem}{Theorem}[section]
\newtheorem{lemma}[theorem]{Lemma}

\theoremstyle{definition}

\theoremstyle{remark}
\newtheorem{remark}[theorem]{Remark}
\numberwithin{equation}{section}

\begin{document}
\title{Long time existence and bounded scalar curvature in the Ricci-harmonic
flow}
\author{Yi Li}
\address{Department of Mathematics, Shanghai Jiao Tong University, 800
Dongchuan Road, Shanghai, 200240 China}
\email{yilicms@gmail.com}

\thanks{Yi Li is partially sponsored by Shanghai Sailing Program (grant) No.
14YF1401400 and NSFC (grant) No. 11401374.}


\maketitle

\begin{abstract} In this paper we study the long time existence of the Ricci-harmonic flow in
terms of scalar curvature and Weyl tensor which extends Cao's result \cite{Cao2011} in the Ricci flow. In dimension four, we also study
the integral bound of the ``Riemann curvature'' for the Ricci-harmonic flow generalizing a recently
result of Simon \cite{Simon2015}.
\end{abstract}

\renewcommand{\labelenumi}{Case \theenumi.} %
\newtheoremstyle{mystyle}{3pt}{3pt}{\itshape}{}{\bfseries}{}{5mm}{} %
\theoremstyle{mystyle} \newtheorem{Thm}{Theorem} \theoremstyle{mystyle} %
\newtheorem{lem}{Lemma} \newtheoremstyle{citing}{3pt}{3pt}{\itshape}{}{%
\bfseries}{}{5mm}{\thmnote{#3}} \theoremstyle{citing} %
\newtheorem*{citedthm}{}

\section{Introduction}\label{section1}

In this paper we study the long time existence of the Ricci-harmonic flow (RHF)
\begin{equation}
\partial_{t}g(t)=-2{\rm Ric}_{g(t)}
+2\alpha(t)\nabla_{g(t)}\phi(t)\otimes\nabla_{g(t)}\phi(t), \ \ \
\partial_{t}\phi(t)=\Delta_{g(t)}\phi(t).\label{1.1}
\end{equation}
introduced in \cite{List2005, List2008, Muller2009, Muller2012}, where $g(t)$ is a family of
Riemannian metric, $\phi(t)$ is a family of functions, and $t\in[0,T)$ (with $T\leq+\infty$). Here $\alpha(t)$ is a time-dependent positive constant. In particular, we may choose $\alpha(t)\equiv\alpha$ a positive constant. If all
functions $\phi(t)\equiv0$, we obtain the Ricci flow (RF) introduced by Hamilton in his
famous paper \cite{H1982} and definitely used by Perelman \cite{P1, P2, P3} on his work about the
Poincar\'e conjecture. The flow equations (\ref{1.1}) come from static Einstein vacuum equations arising in the general
relativity, and also arise as dimensional reductions of RF in higher
dimensions \cite{Lott-Sesum2014}.

For the RF, Hamilton \cite{H1982} showed that the short-time existence and
\begin{equation}
T<\infty\Longrightarrow\limsup_{t\to T}
\left(\max_{M}|{\rm Rm}_{g(t)}|^{2}_{g(t)}\right)=\infty.\label{1.2}
\end{equation}
Later, Sesum \cite{Sesum2005} improved Hamilton's result as
\begin{equation}
T<\infty\Longrightarrow\limsup_{t\to T}
\left(\max_{M}|{\rm Ric}_{g(t)}|^{2}_{g(t)}\right)
=\infty\label{1.3}
\end{equation}
by blow-up argument. For the integral bounds, Ye \cite{Ye2008} and Wang \cite{Wang2008} independently proved that
\begin{equation}
T<\infty\Longrightarrow\left(\int^{T}_{0}
\int_{M}|{\rm Rm}_{g(t)}|^{\frac{m+2}{2}}_{g(t)}dV_{g(t)}dt\right)^{\frac{2}{m+2}}
=\infty.\label{1.4}
\end{equation}
Moreover, Wang \cite{Wang2008} proved another version for RH that
\begin{equation}
{\rm Ric}_{g(t)}\geq-C, \ T<\infty\Longrightarrow
\left(\int^{T}_{0}\int_{M}
|R_{g(t)}|^{\frac{m+2}{2}}_{g(t)}dV_{g(t)}dt\right)^{\frac{2}{m+2}}=\infty.
\label{1.5}
\end{equation}
Here $C$ is a uniform constant. For other works on integral bounds, see for example \cite{Chen-Wang2013,
Ma-Cheng2010, Wang2012, Ye2008a}.

A well-known conjecture (see \cite{Cao2011}) about the extension of RF states that
{\it \begin{equation}
\text{Is it true for} \ \limsup_{t\to T}\left(\max_{M}R_{g(t)}\right)=\infty? \
\text{Here} \ T<\infty.\label{1.6}
\end{equation}}
This conjecture was settled for K\"ahler-Ricci flow by Zhang \cite{Zhang2010} and for type-I maximal
solution of RF by Enders-M\"uller-Topping \cite{Enders-Muller-Topping2011}. Cao \cite{Cao2011} proved the following
\begin{equation}
T<\infty\Longrightarrow\begin{array}{cc}
\text{either} \ \limsup_{t\to T}\left(\max_{M}R_{g(t)}\right)=\infty, \ \text{or}\\
\limsup_{t\to T}\left(\max_{M}R_{g(t)}\right)<\infty \ \text{but} \
\limsup_{t\to T}\frac{|W_{g(t)}|_{g(t)}}{R_{g(t)}}=\infty,
\end{array}\label{1.7}
\end{equation}
where $W_{g(t)}$ denotes the Weyl tensor of $g(t)$.

For $4D$ RF, Simon \cite{Simon2015} and Bamler-Zhang \cite{Bamler-Zhang2015} independently proved
\begin{equation}
T<\infty, \ |R_{g(t)}|\leq C\Longrightarrow
\int_{M}|{\rm Ric}_{g(t)}|^{2}_{g(t)}dV_{g(t)}, \
\int_{M}|{\rm Rm}_{g(t)}|^{2}_{g(t)}dV_{g(t)}\leq C'\label{1.8}
\end{equation}
by different methods (for earlier work see \cite{Tian-Zhang2013}).
\\

On the other hand, for the Ricci-harmonic flow (RHF) we have the following results. When $m=\dim M\geq3$ and
$T<+\infty$. M\"uller \cite{Muller2009, Muller2012} showed that (\ref{1.2}) is also true for RHF. Recently, Cheng and Zhu \cite{CZ2013} extended Sesum's result \cite{Sesum2005} to Ricci-harmonic flow, that is, (\ref{1.3}) is true for RHF. For more results about the RHF, see \cite{Abolarinwa2015, Bailesteanu2013a, Bailesteanu2013b,
Bailesteanu-Tran2013,
Cao-Guo-Tran2015, CZ2013, Fang2013a, Fang2013b, Fang-Zheng2015, Fang-Zheng2016, Fang-Zhu2015, Guo-He2014,
Guo-Huang-Phong2015, Guo-Ishida2014, Guo-Philipowski-Thalmaier2013, Guo-Philipowski-Thalmaier2014,
Guo-Philipowski-Thalmaier2015,Li2014, List2005, List2008, Liu-Wang2014, Muller2009,
Muller2010, Muller2012, Muller-Rupflin2015, Tadano2015a, Tadano2015b, WLF2012, Williams2015, Zhu2013}.

The first main result is an extension of Cao's result \cite{Cao2011} to RHF.

\begin{theorem}\label{t1.1} Let $(M,g(t),\phi(t))_{t\in[0,T)}$ be a solution to RHF (\ref{1.1}) with $\alpha(t)
\equiv\alpha$ a positive constant on a closed manifold $M$ with $m=\dim M\geq3$ and $T<+\infty$. Either one has
\begin{equation*}
\limsup_{t\to T}\left(\max_{M} R_{g(t)}\right)=\infty
\end{equation*}
or
\begin{equation*}
\limsup_{t\to T}
\left(\max_{M}R_{g(t)}\right)<\infty \ \text{but} \
\limsup_{t\to T}
\left(\max_{M}\frac{|W_{g(t)}|_{g(t)}+|\nabla^{2}_{g(t)}
\phi(t)|^{2}_{g(t)}}{R_{g(t)}}\right)=\infty.
\end{equation*}
Here $W_{g(t)}$ is the Weyl part of ${\rm Rm}_{g(t)}$.
\end{theorem}

The second main result focuses on the $4D$ RHF. Write
\begin{equation*}
{\rm Sic}_{g(t)}:={\rm Ric}_{g(t)}
-\alpha\nabla_{g(t)}\phi(t)\otimes\nabla_{g(t)}
\phi(t), \ \ \ S_{g(t)}:=R_{g(t)}-\alpha|\nabla_{g(t)}
\phi(t)|^{2}_{g(t)}.
\end{equation*}
According to Theorem \ref{t2.2} below we can find a uniform constant $C$ such that $S_{g(t)}+C>0$ for all
$t\in[0,T)$.

\begin{theorem}\label{t1.2} Let $(M, g(t),\phi(t))_{t\in[0,T)}$ be a solution to RHF (\ref{1.1}) on a closed manifold $M$
with $m=\dim M=4$, $T\leq+\infty$, $\alpha(t)\equiv\alpha$ a positive constant. Choose a uniform constant
$C$ in
Theorem \ref{t2.2} such that $S_{g(t)}+C>0$. Then
\begin{eqnarray}
\int_{M}|{\rm Sic}_{g(s)}|_{g(s)}dV_{g(s)}&\leq&2c_{0}(M,g(0),\phi(0),s)
+\frac{C}{2}{\rm Vol}(M,g(s))\label{1.9}\\
&&+ \ 1148e^{36Cs}\int^{s}_{0}\int_{M}S^{2}_{g(t)}\!\ dV_{g(t)}dt,\nonumber\\
\int^{s}_{0}\int_{M}|{\rm Sic}_{g(t)}|^{2}_{g(t)}dV_{g(t)}dt&\leq&8 c_{0}(M, g(0),\phi(0),s)
+\frac{C^{2}}{4}\int^{s}_{0}{\rm Vol}(M, g(t))\!\ dt\nonumber\\
&&+ \ 4592 e^{36Cs}\int^{s}_{0}\int_{M}S^{2}_{g(t)}\!\ dV_{g(t)}dt,\label{1.10}
\end{eqnarray}
for all $s\in[0,T)$. Here, $A_{1}=\max_{M}|\nabla_{g(0)}\phi(0)|^{2}_{g(0)}$ and
\begin{eqnarray}
&&c_{0}(M, g(0),\phi(0),s)\nonumber\\
&=&\frac{256\pi^{2}\chi(M)}{36C}\left(e^{36Cs}-1\right)
+\frac{104\alpha^{2}A^{2}_{1}{\rm Vol}(M,g(0))}{35C}\left(e^{35Cs}
-e^{Cs}\right)\label{1.11}\\
&&+ \ e^{37Cs}\alpha A_{1}{\rm Vol}(M, g(0))+e^{36Cs}\int_{M}\frac{|{\rm Sic}_{g(0)}|^{2}_{g(0)}}{S_{g(0)}
+C}dV_{g(0)}.\nonumber
\end{eqnarray}
\end{theorem}

According to Theorem \ref{t2.3} below and following \cite{Simon2015}, we consider the basic assumption
({\bf BA}) for a solution $(M, g(t),\phi(t))_{t\in[0,T)}$ to RHF:
\begin{itemize}

\item[(a)] $M$ is a connected and closed $4$-dimensional smooth manifold,

\item[(b)] $(M, g(t),\phi(t))_{t\in[0,T)}$ is a solution to RHF with $\alpha(t)\equiv\alpha$ a positive constant,

\item[(c)] $T<+\infty$,

\item[(d)] $\max_{M\times[0,T)}|S_{g(t)}|\leq1$.

\end{itemize}
The upper bound $1$ in condition (d) is not essential, since we can rescale the pair
$(g(t), \phi(t))$ so that the condition (d) is always satisfied. Furthermore, since
\begin{equation*}
|\nabla_{g(t)}
\phi(t)|^{2}_{g(t)}\leq A_{1}
\end{equation*}
(by (\ref{3.6})) it follows that condition (d) is equivalent to the uniform bound for $R_{g(t)}$.

\begin{theorem}\label{t1.3} If $(M, g(t),\phi(t))_{t\in[0,T)}$ satisfies {\bf BA}, then
\begin{eqnarray}
\int_{M}|{\rm Sic}_{g(s)}|^{2}_{g(s)}dV_{g(s)}&\leq&b(M,g(0),\phi(0),s),\label{1.12}\\
\int^{s}_{0}\int_{M}|{\rm Sic}_{g(t)}|^{4}_{g(t)}dV_{g(t)}dt&\leq&b(M,g(0),\phi(0),s),\label{1.13}
\end{eqnarray}
for any $s\in[0,T)$. Here
\begin{eqnarray}
b(M,g(0),\phi(0),s)&:=&9e^{88s}\int_{M}|{\rm Sic}_{g(0)}|^{2}_{g(0)}dV_{g(0)}
+\frac{1152}{88}\pi^{2}\chi(M)\left(e^{88s}-1\right)\nonumber\\
&&+ \ \frac{468}{86}(\alpha A_{1})^{2}{\rm Vol}(M,g(0))\left(e^{88s}
-e^{2s}\right)\label{1.14}\\
&&+ \ 9(\alpha A_{1}){\rm Vol}(M,g(0))e^{90s}.\nonumber
\end{eqnarray}
\end{theorem}

Define
\begin{eqnarray}
c(M,g(0),\phi(0),T)&:=&9e^{90T}\bigg[\int_{M}|{\rm Sic}_{g(0)}|^{2}_{g(0)}dV_{g(0)}
+\pi^{2}|\chi(M)|\nonumber\\
&&+ \ [(\alpha A_{1})^{2}+(\alpha A_{1})]{\rm Vol}(M,g(0))\bigg].\label{1.15}
\end{eqnarray}
Then $|b(M,g(0),\phi(0),T)|\leq c(M,g(0),\phi(0),T)$. Theorem \ref{t1.3} now yields

\begin{theorem}\label{t1.4} If $(M,g(t),\phi(t))_{t\in[0,T)}$ satisfies {\bf BA}, then
\begin{eqnarray}
\sup_{t\in[0,T)}\int_{M}|{\rm Sic}_{g(t)}|^{2}_{g(t)}dV_{g(t)}
&\leq&c(M,g(0),\phi(0),T) \ \ < \ \ +\infty,\label{1.16}\\
\sup_{t\in[0,T)}\int_{M}|{\rm Sm}_{g(t)}|^{2}_{g(t)}dV_{g(t)}
&\leq&32\pi^{2}\chi(M)+8c(M,g(0),\phi(0),T)\label{1.17}\\
&&+ \ 13(\alpha A_{1})^{2}{\rm Vol}(M,g(0))e^{2T} \ \ < \ \ +\infty.\nonumber
\end{eqnarray}
\end{theorem}

\begin{remark}\label{r1.5} For a time-dependent tensor field, we always omit time variable in its
components. Although in Theorem \ref{t1.2}--Theorem \ref{t1.4}, we assume that $\alpha(t)\equiv\alpha$ is a positive constant, the same results also hold for $\alpha(t)>0$ and $\dot{\alpha}(t)<0$ (see Section \ref{section4}).
\end{remark}

The proof of Theorem \ref{t1.1} is based on a ``curvature pinching estimate'' for RHF (see Theorem \ref{t2.2}). The new ingredient in the proof of Theorem \ref{t2.2} is an introduction
of ``Riemann curvature tensor'' ${\rm Sm}_{g(t)}$ for RHF, so that we can
express the Weyl tensor $W_{g(t)}$ in terms of ${\rm Sm}_{g(t)}$.

The proofs of Theorem \ref{t1.2}--Theorem \ref{t1.4} follow from the
method of Simon \cite{Simon2015}. As in \cite{Simon2015} we define
\begin{equation*}
Z_{ijk}:=\left(\nabla_{i}S_{jk}\right)(S_{g(t)}+C)
-S_{jk}\left(\nabla_{i}S_{g(t)}\right), \ \ \
Z_{g(t)}:=(Z_{ijk}), \ \ \ f:=\frac{|{\rm Sic}_{g(t)}|^{2}_{g(t)}}{S_{g(t)}
+C}.
\end{equation*}
Analogous to \cite{Simon2015}, we can show that
\begin{eqnarray*}
\frac{d}{dt}\int_{M}f\!\ dV_{g(t)}
&=&\int_{M}\bigg[-2\frac{|Z|^{2}}{(S+C)^{3}}-2f^{2}
+4\frac{{\rm Sm}({\rm Sic},{\rm Sic})}{S+C}-fS\\
&&- \ 2\alpha\left|\Delta\phi\frac{{\rm Sic}}{S+C}
-\nabla^{2}\phi\right|^{2}+2\alpha|\nabla^{2}\phi|^{2}\bigg]dV_{g(t)}.
\end{eqnarray*}
The main difference is the last term on the right-hand side of the above equation. To control
the integral of $|\nabla^{2}\phi|^{2}$ we make use the evolution equation for $|\nabla\phi|^{2}$ (see
(\ref{3.6})) so that
\begin{equation*}
2\int^{t}_{0}\int_{M}|\nabla^{2}\phi|^{2}dV_{g(s)}ds
+\int_{M}|\nabla\phi|^{2}dV_{g(t)}
\leq e^{Ct}A_{1}{\rm Vol}(M,g(0)).
\end{equation*}
The above estimate not only controls the space-time integral of $|\nabla^{2}
\phi|^{2}$, but also an uniform bound for the integral of $|\nabla\phi|^{2}$. These two
estimates play essential role in the following proof. In dimension four, the famous Gauss-Bonnet-Chern
formula (\ref{3.10}) should transform to (\ref{3.13}), where the terms involving $|\nabla\phi|^{2}$ can be
bounded by the above discussion. A modification of \cite{Simon2015} is now applied to the
RHF.

\section{Curvature pinching estimate for RHF}\label{section2}

Consider a solution $(M, g(t),\phi(t))_{t\in[0,T)}$ to RHF with coupling
time-dependent constant $\alpha(t)$
\begin{equation}
\partial_{t}g(t)=-2{\rm Ric}_{g(t)}
+2\alpha(t)\nabla_{g(t)}\phi(t)\otimes\nabla_{g(t)}\phi(t), \ \ \
\partial_{t}\phi(t)=\Delta_{g(t)}\phi(t).\label{2.1}
\end{equation}
Let
\begin{equation*}
\Box_{g(t)}:=\partial_{t}-\Delta_{g(t)}.
\end{equation*}
As in \cite{List2005, List2008, Muller2009, Muller2012} we define the following notions
\begin{eqnarray}
{\rm Sic}_{g(t)}&:=&{\rm Ric}_{g(t)}-\alpha(t)\nabla_{g(t)}\phi(t)\otimes\nabla_{g(t)}\phi(t),
\label{2.2}\\
S_{g(t)}&:=&{\rm tr}_{g(t)}{\rm Ric}_{\phi(t)} \ \ = \ \ R_{g(t)}
-\alpha(t)\left|\nabla_{g(t)}\phi(t)\right|^{2}_{g(t)}.\label{2.3}
\end{eqnarray}
Here $\alpha(t)$ is a family of time-dependent functions. Motivated by RF, we introduce a ``Riemann
curvature'' type for RHF
\begin{equation}
S_{ijk\ell}:=R_{ijk\ell}-\frac{\alpha}{2}\left(g_{j\ell}
\nabla_{i}\phi\nabla_{k}\phi+g_{k\ell}\nabla_{i}\phi\nabla_{j}\phi\right).\label{2.4}
\end{equation}
Our notation for $S_{ijk\ell}$ implies that
\begin{equation*}
S_{ij}:=g^{k\ell}S_{ik\ell j}=R_{ij}
-\alpha\nabla_{i}\phi\nabla_{j}\phi=g^{k\ell}S_{kij\ell}
\end{equation*}
which coincides with the components of ${\rm Sic}_{g(t)}$. The corresponding tensor field for
$S_{ijk\ell}$ is denoted by ${\rm Sm}_{g(t)}$.

\begin{lemma}\label{l2.1} For a solution $(M, g(t),\phi(t))_{t\in[0,T)}$ with a time-dependent coupling function $\alpha(t)$, we have
\begin{eqnarray}
\Box_{g(t)}S_{g(t)}&=&2|{\rm Sic}_{g(t)}|^{2}_{g(t)}
+2\alpha(t)\left|\Delta_{g(t)}\phi(t)\right|^{2}_{g(t)}
-\dot{\alpha}(t)\left|\nabla_{g(t)}\phi(t)\right|^{2}_{g(t)},\label{2.5}\\
\Box_{g(t)}{\rm Sic}_{g(t)}&=&2{\rm Sm}_{g(t)}({\rm Sic}_{g(t)},\cdot)
-2{\rm Sic}^{2}_{g(t)}\nonumber\\
&&+ \ 2\alpha(t)\Delta_{g(t)}\phi(t)\nabla^{2}_{g(t)}\phi(t)-\dot{\alpha}(t)\nabla_{g(t)}
\phi(t)\otimes\nabla_{g(t)}\phi(t)\label{2.6}
\end{eqnarray}
where ${\rm Sic}^{2}_{g(t)}=(S_{ik}S_{j\ell}g^{k\ell})_{ij}$ and ${\rm Sm}_{g(t)}
({\rm Sic}_{g(t)},\cdot)=(S_{kij\ell}S^{k\ell})_{ij}$.
\end{lemma}

\begin{proof} The first equation can be found in \cite{Muller2012}, Corollary 4.5. In the same corollary we also have
\begin{equation*}
\partial_{t}S_{ij}=\Delta_{g(t),L}S_{ij}+2\alpha\Delta_{g(t)}\phi(t)\nabla_{i}\nabla_{j}\phi
-\dot{\alpha}\nabla_{i}\phi\nabla_{j}\phi.
\end{equation*}
Here $\Delta_{g(t),L}$ denotes the Lichnerowicz Laplacian with respect to $g(t)$ defined by
\begin{equation*}
\Delta_{g(t),L}S_{ij}=\Delta_{g(t)}S_{ij}+2R_{kij\ell}S^{k\ell}
-R_{ik}S_{j}{}^{k}-R_{jk}S_{i}{}^{k}.
\end{equation*}
Then
\begin{equation*}
\Box_{g(t)}S_{ij}=2R_{kij\ell}S^{k\ell}-R_{ik}S_{j}{}^{k}-R_{jk}S_{i}{}^{k}
+2\alpha\Delta_{g(t)}\phi(t) \nabla_{i}\nabla_{j}\phi-\dot{\alpha}\nabla_{i}\phi\nabla_{j}\phi.
\end{equation*}
Plugging $S_{ij}=R_{ij}-\alpha\nabla_{i}\phi\nabla_{j}\phi$ into the above equation yields
the second desired equation.
\end{proof}

As a corollary of Lemma \ref{2.1} we have (see \cite{Muller2012}, Corollary 5.2)
\begin{equation}
\min_{M}S_{g(t)}\geq\min_{M}S_{g(0)}, \ \ \ \text{provided} \ \alpha(t)\geq0 \ \text{and} \
\dot{\alpha}(t)\leq0.\label{2.7}
\end{equation}

\begin{theorem}\label{t2.2}{\bf (Curvature pinching estimate)} Let $(M, g(t), \phi(t))_{t\in[0,T)}$ be a
solution to RHF on a closed manifold $M$ with $m=\dim M\geq3$, $T\leq+\infty$, $\alpha(t)\geq0$ and $\dot{\alpha}(t)\leq0$. There exist uniform constants $C_{1}, C_{2}, C$, depending only on $m,g(0),\phi(0),\alpha(0)$ such that
\begin{equation}
S_{g(t)}+C>0, \ \ \ \frac{|{\rm Sin}_{g(t)}|_{g(t)}}{S_{g(t)}+C}
\leq C_{1}+C_{2}\max_{M\times[0,t]}\sqrt{\frac{|W_{g(s)}|_{g(s)}+|\nabla^{2}_{g(s)}
\phi(s)|^{2}_{g(s)}}{S_{g(s)}+C}}\label{2.8}
\end{equation}
where ${\rm Sin}_{g(t)}={\rm Sic}_{g(t)}-\frac{S_{g(t)}}{m}g(t)$ is the trace-free part of
${\rm Sic}_{g(t)}$ and $W_{g(t)}$ is the Weyl tensor field of $g(t)$.
\end{theorem}

\begin{proof} The first inequality follows from (\ref{2.7}). Let
\begin{equation*}
f:=\frac{|{\rm Sin}_{g(t)}|^{2}_{g(t)}}{(S_{g(t)}+C)^{\gamma}}, \ \ \ \gamma>0.
\end{equation*}
Since ${\rm Sin}_{g(t)}=({\rm Sic}_{g(t)}+\frac{C}{m}g(t))-(\frac{S_{g(t)}+C}{m})g(t)$, it follows that
\begin{equation*}
f=\frac{|{\rm Sic}_{g(t)}+\frac{C}{m}g(t)|^{2}_{g(t)}}{(S_{g(t)}+C)^{\gamma}}
-\frac{1}{m}(S_{g(t)}+C)^{2-\gamma}.
\end{equation*}
In the following we always omit the subscripts $t$ and $g(t)$ and set
\begin{equation*}
{\rm Sic}'_{g(t)}:={\rm Sic}_{g(t)}+\frac{C}{m}g(t), \ \ \
S'_{g(t)}:=S_{g(t)}+C={\rm tr}_{g(t)}{\rm Sic}'_{g(t)}.
\end{equation*}
In the following we always omit the subscripts $g(t)$ and $t$. Using the equation (3.21) in \cite{CLN2006} we have
\begin{eqnarray*}
\Box\frac{|{\rm Sic}'|^{2}}{(S')^{\gamma}}&=&\frac{1}{(S')^{\gamma}}\Box|{\rm Sic}'|^{2}
-\gamma\frac{|{\rm Sic}'|^{2}}{(S')^{\gamma+1}}\Box S-\gamma(\gamma+1)\frac{|{\rm Sic
}'|^{2}}{(S')^{\gamma+2}}|\nabla
S'|^{2}\\
&&+ \ \frac{2\gamma}{(S')^{\gamma+1}}\langle\nabla
|{\rm Sic}'|^{2},
\nabla S'\rangle.
\end{eqnarray*}
It is clear from (\ref{2.6}) that
\begin{equation*}
\Box|{\rm Sic}|^{2}=-2\left|\nabla{\rm Sic}\right|^{2}
+4{\rm Sm}\left({\rm Sic},{\rm Sic}\right)+2\bigg\langle{\rm Sic},2\alpha\Delta\phi\nabla^{2}\phi
-\dot{\alpha}\nabla\phi
\otimes\nabla\phi(t)\bigg\rangle.
\end{equation*}
Therefore
\begin{eqnarray*}
\Box|{\rm Sic}'|^{2}&=&\Box\left[|{\rm Sic}|^{2}
+\frac{C^{2}}{m}+\frac{2C}{m}S\right]\\
&=&\Box|{\rm Sic}|^{2}+\frac{2C}{m}
\left[2|{\rm Sic}|^{2}+2\alpha\left|\Delta\phi\right|^{2}
-\dot{\alpha}\left|\nabla\phi\right|^{2}\right]\\
&=&-2|\nabla{\rm Sic}'|^{2}+4{\rm Sm}({\rm Sic},{\rm Sic})+\frac{4C}{m}|{\rm Sic}|^{2}\\
&&+ \ 2\left\langle{\rm Sic}',2\alpha\Delta\phi
\nabla^{2}_\phi-\dot{\alpha}\nabla
\phi\otimes\nabla\phi\right\rangle\\
&=&-2|\nabla{\rm Sic}'|^{2}+4{\rm Sm}
({\rm Sic}',{\rm Sic}')-\frac{4C}{m}|{\rm Sic}'|^{2}
+\frac{4C^{2}}{m^{2}}(S-C)\\
&&+ \ 2\langle{\rm Sic}'
,2\alpha\Delta\phi\nabla^{2}
\phi-\dot{\alpha}\nabla\phi
\otimes\nabla\phi\rangle
\end{eqnarray*}
and
\begin{eqnarray*}
\Box\frac{|{\rm Sic}'|^{2}}{(S')^{\gamma}}
&=&-\frac{2}{(S')^{\gamma}}
|\nabla{\rm Sic}'|^{2}-
\frac{2\gamma|{\rm Sic}'|^{4}}{(S')^{\gamma+1}}+\frac{4}{(S')^{\gamma}}{\rm Sm}({\rm Sic}',
{\rm Sic}')\\
&&- \ \gamma(\gamma+1)\frac{|{\rm Sic}'|^{2}|\nabla S|^{2}}{(S')^{\gamma+2}}+\frac{2\gamma}{(S')^{\gamma+1}}\langle\nabla|{\rm Sic}'|^{2},\nabla
S'\rangle\\
&&+ \ \frac{4C^{2}}{m^{2}}\frac{S'-2C}{(S')^{\gamma}}
-\frac{4C}{m}\frac{(1+\gamma)S'-2\gamma C}{(S')^{\gamma+1}}|{\rm Sic}'|^{2}\\
&&+ \ \frac{2}{(S')^{\gamma}}\langle{\rm Sic}',
2\alpha\Delta\phi\nabla^{2}\phi-\dot{\alpha}\nabla\phi\otimes\nabla\phi\rangle\\
&&- \ \frac{\gamma|{\rm Sic}'|^{2}}{(S')^{\gamma+1}}
\left[2\alpha|\Delta\phi|^{2}-\dot{\alpha}|\nabla\phi|^{2}\right].
\end{eqnarray*}
Using the identity
\begin{equation*}
\left\langle\nabla\frac{|{\rm Sic}'|^{2}}{(S')^{\gamma}},\nabla S'\right\rangle
=
\frac{1}{(S')^{\gamma}}\langle\nabla|{\rm Sic}'|^{2},\nabla S'\rangle-\frac{\gamma}{(S')^{\gamma+1}}
|\nabla S|^{2}|{\rm Sic}'|^{2},
\end{equation*}
we arrive at
\begin{eqnarray*}
\Box\frac{|{\rm Sic}'|^{2}}{(S')^{\gamma}}
&=&\frac{2\gamma}{S'}\left\langle\nabla\frac{|{\rm Sic}'|^{2}}{(S')^{\gamma}},
\nabla S'\right\rangle-\frac{2}{(S')^{\gamma+2}}|S'\nabla{\rm Sic}'|^{2}\\
&&+ \ \frac{\gamma(\gamma-1)}{(S')^{\gamma+2}}|{\rm Sic}'|^{2}|\nabla S'|^{2}
+\frac{4}{(S')^{\gamma}}{\rm Sm}({\rm Sic}',{\rm Sic}')-\frac{2\gamma}{(S')^{\gamma+1}}
|{\rm Sic}'|^{4}\\
&&+ \ \frac{4C^{2}}{m^{2}}\frac{S'-2C}{(S')^{\gamma}}
-\frac{4C}{m}\frac{(1+\gamma)S'-2\gamma C}{(s')^{\gamma+1}}|{\rm Sic}'|^{2}\\
&&+ \ \frac{2}{(S')^{\gamma}}
\langle{\rm Sic}',2\alpha\Delta\phi\nabla^{2}\phi-\dot{\alpha}
\nabla\phi\otimes\nabla\phi\rangle\\
&&- \ \frac{\gamma|{\rm Sic}'|^{2}}{(S')^{\gamma+1}}\left[2\alpha|\Delta\phi|^{2}
-\dot{\alpha}|\nabla\phi|^{2}\right].
\end{eqnarray*}
On the other hand, the following two identities
\begin{eqnarray*}
|S'\nabla{\rm Sic}'|^{2}&=&|S'\nabla_{i}S'_{jk}-S'_{jk}\nabla_{i}S'+S'_{jk}\nabla_{i}S'|^{2}\\
&=&|S'\nabla_{i}S'_{jk}-S'_{jk}\nabla_{i}S'|^{2}-|{\rm Sic}'|^{2}|\nabla S'|^{2}
+S'\langle\nabla|{\rm Sic}'|^{2},\nabla S'\rangle,\\
\langle\nabla|{\rm Sic}'|^{2},\nabla S'\rangle&=&\left\langle\nabla\left[(S')^{\gamma}
\frac{|{\rm Sic}'|^{2}}{(S')^{\gamma}}\right],\nabla S'\right\rangle\\
&=&\frac{\gamma}{S'}|{\rm Sic}'|^{2}|\nabla S'|^{2}
+(S')^{\gamma}\left\langle\nabla\frac{|{\rm Sic}'|^{2}}{(S')^{\gamma}},\nabla S'\right\rangle
\end{eqnarray*}
implies
\begin{eqnarray*}
\Box\frac{|{\rm Sic}'|^{2}}{(S')^{\gamma}}&=&\frac{2(\gamma-1)}{S'}
\left\langle\nabla\frac{|{\rm Sic}'|^{2}}{(S')^{\gamma}},\nabla S'\right\rangle
-\frac{2}{(S')^{\gamma+2}}\left|S'\nabla_{i}S'_{jk}-S'_{jk}\nabla_{i}S'\right|^{2}\\
&&- \ \frac{(2-\gamma)(\gamma-1)}{(S')^{\gamma+2}}|{\rm Sic}'|^{2}|\nabla S'|^{2}
+\frac{4}{(S')^{\gamma}}{\rm Sm}({\rm Sic}',{\rm Sic}')
-\frac{2\gamma|{\rm Sic}'|^{4}}{(S')^{\gamma+1}}\\
&&+ \ \frac{4C^{2}}{m^{2}}\frac{S'-2C}{(S')^{\gamma}}-\frac{4C}{m}
\frac{(1+\gamma)S'-2\gamma C}{(S')^{\gamma+1}}|{\rm Sic}'|^{2}\\
&&+ \ \frac{2}{(S')^{\gamma}}\langle{\rm Sic}',
2\alpha\Delta\phi\nabla^{2}\phi-\dot{\alpha}\nabla\phi\otimes\nabla\phi\rangle\\
&&- \ \frac{\gamma|{\rm Sic}'|^{2}}{(S')^{\gamma+1}}\left[2\alpha|\Delta\phi|^{2}
-\dot{\alpha}|\nabla\phi|^{2}\right].
\end{eqnarray*}
It is clear that $\Box(S')^{2-\gamma}=(2-\gamma)(S')^{1-\gamma}\Box S'
-(2-\gamma)(1-\gamma)(S')^{-\gamma}|\nabla S'|^{2}$ and $|{\rm Sic}|^{2}=|{\rm Sic}'|^{2}
+\frac{C^{2}}{m}-\frac{2CS'}{m}$. These yield
\begin{eqnarray*}
\Box(S')^{2-\gamma}&=&\frac{2(\gamma-1)}{S'}\left\langle\nabla(S')^{2-\gamma},\nabla S'\right\rangle
+2(2-\gamma)(S')^{1-\gamma}|{\rm Sic}'|^{2}\\
&&- \ \frac{(2-\gamma)(\gamma-1)}{(S')^{2}}(S')^{2-\gamma}|\nabla S'|^{2}+(2-\gamma)(S')^{1-\gamma}
\left[\frac{2C^{2}}{m}-\frac{4C}{m}S'\right]\\
&&+ \ (2-\gamma)(S')^{1-\gamma}
\left[2\alpha|\Delta\phi|^{2}-\dot{\alpha}|\nabla\phi|^{2}\right].
\end{eqnarray*}
Then
\begin{eqnarray}
\Box f&=&\frac{2(\gamma-1)}{S'}\langle\nabla f,\nabla S'\rangle
-\frac{2}{(S')^{\gamma+2}}|S'\nabla_{i}S'_{jk}-S'_{jk}\nabla_{i}S'|^{2}\nonumber\\
&&- \ \frac{(2-\gamma)(\gamma-1)}{(S')^{2}}|\nabla S'|^{2}f+\mathscr{Q}_{1}+\mathscr{Q}_{2}
+\mathscr{Q}_{3}\label{2.9}
\end{eqnarray}
where
\begin{eqnarray}
\mathscr{Q}_{1}&:=&-\frac{2(2-\gamma)}{m}(S')^{1-\gamma}|{\rm Sic}'|^{2}
+\frac{4}{(S')^{\gamma}}{\rm Sm}({\rm Sic}',{\rm Sic'})-\frac{2\gamma|{\rm Sic}'|^{4}}{(S')^{
1+\gamma}},\label{2.10}\\
\mathscr{Q}_{2}&:=&\frac{4C}{m}\left[\frac{C(S'-2C)}{m(S')^{2}}
-\frac{(1+\gamma)S'-2\gamma C}{(S')^{\gamma+1}}|{\rm Sic}'|^{2}-\frac{2-\gamma}{2m}
\frac{C-2S'}{(S')^{\gamma-1}}\right]\label{2.11}\\
\mathscr{Q}_{3}&:=&\frac{2}{(S')^{\gamma}}\langle{\rm Sic}',2\alpha\Delta\phi\nabla^{2}\phi
-\dot{\alpha}\nabla\phi\otimes\nabla\phi\rangle\nonumber\\
&&- \ \frac{2\alpha|\Delta\phi|^{2}-\dot{\alpha}|\nabla\phi|^{2}}{(S')^{\gamma+1}}
\left[\gamma|{\rm Sic}'|^{2}+\frac{2-\gamma}{m}(S')^{2}\right].\label{2.12}
\end{eqnarray}
Observe that
\begin{equation}
\mathscr{Q}_{2}=\frac{4C}{m^{2}}\left[\frac{C}{S'}-\frac{2C^{2}}{(S')^{2}}
+\frac{(3\gamma-2)C}{2(S')^{\gamma-1}}+\frac{1-\gamma}{(S')^{\gamma-2}}
+\frac{f}{m}\left(\frac{2\gamma C}{S'}-1-\gamma\right)\right].\label{2.13}
\end{equation}
The $\mathscr{Q}_{1}$ terms can be written as
\begin{eqnarray}
\mathscr{Q}_{1}&=&\frac{2}{(S')^{\gamma+1}}\left[\frac{\gamma-2}{m}|{\rm Sic}'|^{2}
(s')^{2}+2S'{\rm Sm}({\rm Sic}',{\rm Sic}')-\gamma|{\rm Sic}'|^{4}\right]\nonumber\\
&=&\frac{2}{(S')^{\gamma+1}}\bigg[(2-\gamma)|{\rm Sic}'|^{2}|{\rm Sin}|^{2}
-2\left(|{\rm Sic}'|^{4}-S'{\rm Sm}({\rm Sic}',{\rm Sic}')\right)\bigg],\label{2.14}
\end{eqnarray}
where ${\rm Sin}={\rm Sic}-\frac{S}{m}g={\rm Sic}'-\frac{C}{m}g-\frac{S}{m}g
={\rm Sic}'-\frac{S'}{m}g=:{\rm Sin}'$. Recall the decomposition of ${\rm Rm}$ that
\begin{equation*}
R_{ijk\ell}=\frac{1}{m-2}(R_{i\ell}g_{jk}
+R_{jk}g_{i\ell}-R_{ik}g_{j\ell}-R_{j\ell}g_{ik})
-\frac{R(g_{i\ell}g_{jk}-g_{ik}g_{j\ell})}{(m-1)(m-2)}
+W_{ijk\ell}
\end{equation*}
where $W_{ijk\ell}$ stands for the Weyl tensor field. Then
\begin{eqnarray}
S_{ijk\ell}&=&W_{ijk\ell}+\frac{1}{m-2}(S'_{i\ell}g_{jk}
+S'_{jk}g_{i\ell}
-S'_{ik}g_{j\ell}-S'_{j\ell}g_{ik})-\frac{S'(g_{i\ell}g_{jk}
-g_{ik}g_{j\ell})}{(m-2)(m-1)}\nonumber\\
&&+ \ \frac{C}{m(m-1)(m-2)}(g_{i\ell}g_{jk}
-g_{ik}g_{j\ell})-\frac{C}{m(m-2)}(g_{jk}g_{i\ell}
-g_{j\ell}g_{ik})\nonumber\\
&&+ \ \frac{\alpha}{m-2}(g_{jk}\nabla_{i}\phi\nabla_{\ell}\phi
+g_{i\ell}\nabla_{j}\phi\nabla_{k}\phi-g_{j\ell}\nabla_{i}\phi\nabla_{k}\phi
-g_{ik}\nabla_{j}\phi\nabla_{\ell}\phi)\label{2.15}\\
&&- \ \frac{\alpha|\nabla\phi|^{2}}{(m-2)(m-1)}(g_{i\ell}g_{jk}
-g_{ik}g_{j\ell})-\frac{\alpha}{2}(g_{j\ell}\nabla_{i}\phi\nabla_{k}
\phi+g_{k\ell}\nabla_{i}\phi\nabla_{j}\phi).\nonumber
\end{eqnarray}
Together with (\ref{2.15}), we get
\begin{eqnarray}
{\rm Sm}({\rm Sic}',{\rm Sic}')&=&\frac{1}{m-2}
\bigg[\frac{2m-1}{m-1}S'|{\rm Sic}'|^{2}-2{\rm Sic}'{}^{3}-\frac{S'{}^{3}}{m-1}\bigg]
+W({\rm Sic}',{\rm Sic}')\nonumber\\
&&- \ \frac{1}{m-1}\left(\frac{C}{m}+\frac{\alpha}{m-2}|\nabla\phi|^{2}\right)
\left(|S'|^{2}-|{\rm Sic}'|^{2}\right)\label{2.16}\\
&&+ \ \frac{2\alpha}{m-2}\left\langle S'{\rm Sic}'-\frac{m}{2}{\rm Sic}'{}^{2},
\nabla\phi\otimes\nabla\phi\right\rangle\nonumber
\end{eqnarray}
where ${\rm Sic}'^{3}=S'_{ij}S'{}^{j}{}_{k}S'{}^{ki}$. Substituting (\ref{2.16}) into (\ref{2.14}) we arrive at
\begin{eqnarray}
\mathscr{Q}_{1}&=&\frac{2}{(S')^{\gamma+1}}\bigg[(2-\gamma)|{\rm Sic}'|^{2}|{\rm Sin}|^{2}
-2\mathscr{Q}_{4}+2S'W({\rm Sic}',{\rm Sic}')\nonumber\\
&&- \ \frac{2}{m-1}\left(\frac{C}{m}+\frac{\alpha}{m-2}|\nabla\phi|^{2}\right)
\left(S'{}^{3}-S'|{\rm Sic}'|^{2}\right)\label{2.17}\\
&&+ \ \frac{2\alpha}{m-2}\left\langle S'{}^{2}{\rm Sic}'
-\frac{m}{2}S'{\rm Sic}'{}^{2},\nabla\phi\otimes\nabla\phi\right\rangle\bigg]\nonumber
\end{eqnarray}
where
\begin{equation}
\mathscr{Q}_{4}=|{\rm Sic}'|^{4}
-\frac{S'}{m-2}\left(\frac{2m-1}{m-1}S'|{\rm Sic}'|^{2}-2{\rm Sic}'{}^{3}
-\frac{S'{}^{3}}{m-1}\right).\label{2.18}
\end{equation}
The first term
\begin{equation*}
(2-\gamma)|{\rm Sic}'|^{2}|{\rm Sin}|^{2}
-2\mathscr{Q}_{4}+2S'W({\rm Sic}',{\rm Sic}')
\end{equation*}
on the right-hand side of (\ref{2.17}) can be written as
\begin{eqnarray*}
&&-\gamma|{\rm Sin}|^{4}+\left[\frac{2(m^{2}+2m-2)}{m(m-1)(m-2)}
-\frac{\gamma}{m}\right]S'{}^{2}|{\rm Sin}|^{2}\\
&& \ \ \ \ \ \ \ \ \ \ \ \ \ \ \ \ \ \ \ \ \ \ \ \ \ \ \ \ \ \ \ \
\ \ \ \ \ \ \ \ +\frac{4}{m^{2}(m-2)}S'{}^{4}
-\frac{4}{m-2}S'{\rm Sic}'{}^{3}+2S'W({\rm Sic}',{\rm Sic}')
\end{eqnarray*}
by using the fact that $|{\rm Sic}'|^{2}=|{\rm Sin}|^{2}+\frac{S'{}^{2}}{m}$, where ${\rm Sic}'{}^{3}$ is equals to
\begin{equation*}
{\rm Sic}'{}^{3}
={\rm Sin}^{3}+\frac{3S'}{m}|{\rm Sin}|^{2}+\frac{S'{}^{3}}{m^{2}},
\end{equation*}
and therefore
\begin{equation*}
-\gamma|{\rm Sin}|^{4}+\left[\frac{2(m-2)}{m(m-1)}-\frac{\gamma}{m}\right]
S'{}^{2}|{\rm Sin}|^{2}-\frac{4}{m-2}S'{\rm Sin}^{3}
+2S'W({\rm Sin},{\rm Sin}),
\end{equation*}
because of $W({\rm Sic}',{\rm Sic}')=W({\rm Sin},{\rm Sin})$. Finally, we obtain
\begin{eqnarray}
\Box f&=&2(\gamma-1)\langle\nabla f,\nabla\ln S'\rangle
-\frac{2}{(S')^{\gamma}}|\nabla_{i}S'_{jk}-S'_{jk}\nabla_{i}\ln S'|^{2}\nonumber\\
&&- \ (2-\gamma)(\gamma-1)|\nabla\ln S'|^{2}f+\mathscr{Q}_{1}+\mathscr{Q}_{2}
+\mathscr{Q}_{3}\label{2.19}
\end{eqnarray}
where
\begin{eqnarray*}
\mathscr{Q}_{1}&=&\frac{2}{(S')^{\gamma+1}}
\bigg[-\gamma(S')^{2\gamma}f^{2}+\left(\frac{2m-4}{n(m-1)}-\frac{\gamma}{m}\right)(S')^{\gamma+2}f
-\frac{4S'{}^{4}}{m-2}\frac{\sin^{3}}{S'{}^{3}}\\
&&+ \ 2(S')^{3}W\left(\frac{{\rm Sin}}{S'},\frac{{\rm Sin}}{S'}\right)
-\frac{2}{m-1}\left(\frac{C}{m}+\frac{\alpha|\nabla\phi|^{2}}{m-2}\right)
\left(\frac{m-1}{m}S'{}^{3}
-(S')^{\gamma+1}f\right)\\
&&+ \ \frac{2\alpha}{m-2}\left\langle S'{}^{2}{\rm Sic}'-\frac{m}{2}S'{\rm Sic}'{}^{2},
\nabla\phi\otimes\nabla\phi\right\rangle\bigg],\\
\mathscr{Q}_{2}&=&\frac{4C}{m^{2}}
\bigg[\frac{C}{S'}-\frac{2C^{2}}{S'{}^{2}}+\frac{(3\gamma-2)C}{2(S')^{\gamma+1}}
+\frac{1-\gamma}{(S')^{\gamma-2}}+\frac{f}{m}\left(\frac{2\gamma C}{S'}-1-\gamma\right)\bigg],\\
\mathscr{Q}_{3}&=&\frac{2}{(S')^{\gamma}}\langle{\rm Sic}',2\alpha\Delta\phi
\nabla^{2}\phi-\dot{\alpha}\nabla\phi\otimes\nabla\phi\rangle\\
&&- \ \frac{2\alpha|\Delta\phi|^{2}-\dot{\alpha}|\nabla\phi|^{2}}{(S')^{\gamma+1}}
\left[\gamma(S')^{\gamma}f+\frac{2S'{}^{2}}{m}\right].
\end{eqnarray*}
from (\ref{2.9})--(\ref{2.13}) and (\ref{2.17})--(\ref{2.18}).

In particular, when $\gamma=2$, we have $f=|{\rm Sin}|^{2}/S'{}^{2}=|{\rm Sic}'|^{2}/S'{}^{2}$ and
\begin{equation}
\Box f=2\langle\nabla f,\nabla\ln S'\rangle-2\left|\nabla\left(\frac{{\rm Sin}}{S'}\right)\right|^{2}+\mathscr{Q}_{1}+\mathscr{Q}_{2}+\mathscr{Q}_{3}\label{2.20}
\end{equation}
where
\begin{eqnarray*}
\mathscr{Q}_{1}&=&4S'\bigg[-f^{2}-\frac{f}{m(m-1)}
-\frac{2}{m-2}\frac{{\rm Sin}^{3}}{S'{}^{3}}
+\frac{1}{S'}W\left(\frac{{\rm Sin}}{S'},\frac{{\rm Sin}}{S'}\right)\\
&&- \ \ \frac{1}{S'}\left(\frac{C}{m}+\frac{\alpha|\nabla\phi|^{2}}{m-2}\right)
\left(\frac{1}{m}-\frac{f}{m-1}\right)+\frac{1}{S'}\frac{\alpha}{m-2}\left\langle\frac{{\rm Sin}^{2}}{S'{}^{2}},
\nabla\phi\otimes\nabla\phi\right\rangle\\
&&+ \ \frac{1}{S'}\frac{\alpha|\nabla\phi|^{2}}{2m(m-2)}\bigg],\\
\mathscr{Q}_{2}&=&\frac{4C}{m^{2}}\left[\frac{C}{S'}
-\frac{2C^{2}}{S'{}^{2}}+\frac{2C}{S'{}^{3}}-1+\frac{f}{m}\left(\frac{4C}{S'}-3\right)\right],\\
\mathscr{Q}_{3}&=&\frac{2}{S'}\bigg[\langle{\rm Sin},2\alpha\Delta\phi\nabla^{2}\phi
-\dot{\alpha}\nabla\phi
\otimes\nabla\phi\rangle-\left(2\alpha|\Delta\phi|^{2}-\dot{\alpha}|\nabla\phi|^{2}\right)f\bigg].
\end{eqnarray*}
Since $\alpha(t)\geq0$ and $\dot{\alpha}(t)\leq0$, it follows from \cite{Muller2012}, Proposition 5.5,
that $|\nabla\phi|^{2}$ is bounded from above by a uniform constant. Consequently
\begin{eqnarray*}
\mathscr{Q}_{1}&\leq&4S'\left(-f^{2}-\frac{f}{m(m-1)}+\frac{2}{m-2}f^{3/2}+C_{m}\frac{|W|}{S'}f+C_{0}f\right),\\
\mathscr{Q}_{2}&\leq&4S'(C_{0}+C_{0}f),\\
\mathscr{Q}_{3}&\leq&4S'\left(C_{0}+C_{0}\frac{|\nabla^{2}\phi|^{2}}{S'}f^{1/2}\right)
\end{eqnarray*}
where $C_{m}=C(m)>0$ and $C_{0}=C(m,g(0),\phi(0),\alpha(0))>0$. Without loss of generality, we may assume that $f\geq1$. In this case, we have
\begin{eqnarray}
\Box f&\leq&2\langle\nabla f,\nabla\ln S'\rangle
+4S'f\bigg[-f-\frac{1}{m(m-1)}+\frac{2}{m-2}f^{1/2}\nonumber\\
&&+ \ C_{0}+C_{0}\frac{|W|+|\nabla^{2}\phi|^{2}}{S'}\bigg]\label{2.21}
\end{eqnarray}
Applying the maximum principle to (\ref{2.21}) yields
\begin{equation*}
f-\frac{2}{m-2}f^{1/2}+\frac{1}{m(m-1)}-C_{0}-C_{0}\frac{|W|+|\nabla^{2}\phi|^{2}}{S'}
\leq0
\end{equation*}
at the point where $f$ achieves its maximum. Thus we obtain (\ref{2.8}).
\end{proof}

As an immediate consequently of Theorem \ref{t2.2} we obtain the following theorem that is an extension
of Cao's result (\cite{Cao2011}, Corollary 3.1).

\begin{theorem}\label{t2.3} Let $(M,g(t),\phi(t))_{t\in[0,T)}$ be a solution to RHF with nonincreasing $\alpha(t)
\in[\underline{\alpha}, \overline{\alpha}]$, $0<\underline{\alpha}
\leq\overline{\alpha}<+\infty$, on a closed manifold $M$ with $m=\dim
M\geq3$ and $T<+\infty$. Either one has
\begin{equation*}
\limsup_{t\to T}\left(\max_{M}R_{g(t)}\right)=\infty
\end{equation*}
or
\begin{equation*}
\limsup_{t\to T}\left(\max_{M}R_{g(t)}\right)<\infty \ \text{but} \ \limsup_{t\to T}\left(\max_{M}\frac{|W_{g(t)}|_{g(t)}+|\nabla^{2}_{g(t)}\phi(t)|^{2}_{g(t)}}{R_{g(t)}}
\right)=\infty.
\end{equation*}
\end{theorem}

\begin{proof} Suppose now that
\begin{equation*}
\limsup_{t\to T}\left(\max_{M}R_{g(t)}\right)<\infty \ \text{and} \ \limsup_{t\to T}\left(\max_{M}\frac{|W_{g(t)}|_{g(t)}+|\nabla^{2}_{g(t)}\phi(t)|^{2}_{g(t)}}{R_{g(t)}}
\right)<\infty.
\end{equation*}
In this case both $R_{g(t)}$ and $|W_{g(t)}|_{g(t)}
+|\nabla^{2}_{g(t)}\phi(t)|^{2}_{g(t)}$ are uniformly bounded. Theorem \ref{2.2} then implies that
$|{\rm Sin}_{g(t)}|_{g(t)}$ is uniformly bounded. Since ${\rm Sin}_{g(t)}
={\rm Sic}_{g(t)}-\frac{S_{g(t)}}{m}g(t)$, it follows that ${\rm Sic}_{g(t)}$ is uniformly bounded.
However, the assumption on $\alpha(t)$ tells us that $|\nabla_{g(t)}
\phi(t)|^{2}_{g(t)}$ is uniformly bounded (e.g., \cite{Muller2012}, Proposition 5.5). Thus
${\rm Ric}_{g(t)}$ is uniformly bounded, contradicting with the fact (\ref{1.2}). Therefore we prove the theorem.
\end{proof}

\section{4D Ricci-harmonic flow with bounded $S_{g(t)}$: I}\label{section3}

Let the constant $C$ be given in Theorem \ref{t2.2} and we assume that $\alpha$ is a positive constant so
that $\dot{\alpha}(t)\equiv0$, and $m=\dim M=4$. As in \cite{Simon2015} we define
\begin{equation}
Z_{ijk}:=\left(\nabla_{i}S_{jk}\right)(S_{g(t)}+C)-S_{jk}\left(\nabla_{i}S_{g(t)}
\right), \ \ \ Z_{g(t)}:=(Z_{ijk}).\label{3.1}
\end{equation}
In the proof of Theorem \ref{t2.2}, we actually have proved
\begin{eqnarray}
\Box\frac{|{\rm Sic}|^{2}}{S+C}
&=&-2\frac{|Z|^{2}}{(S+C)^{3}}
-2\frac{|{\rm Sic}|^{4}}{(S+C)^{2}}
+4\frac{{\rm Sm}({\rm Sic},{\rm Sic})}{S
+C}\nonumber\\
&&- \ \frac{1}{(S+C)^{2}}
\bigg[\left(2\alpha|\Delta\phi|^{2}
-\dot{\alpha}|\nabla\phi|^{2}
\right)|{\rm Sic}|^{2}\label{3.2}\\
&&- \ 2(S+C)
\left\langle{\rm Sic},2\alpha\Delta\phi
\nabla^{2}\phi-\dot{\alpha}\nabla
\phi\otimes\nabla\phi\right\rangle
\bigg].\nonumber
\end{eqnarray}
The bracket in the right-hand side of (\ref{3.2}) can be expressed as
\begin{equation*}
2\alpha\left[\left|{\rm Sic}\!\ \Delta\phi-(S+C)\nabla^{2}\phi\right|^{2}
-(S+C)^{2}|\nabla^{2}\phi|^{2}\right]
\end{equation*}
because of $\dot{\alpha}\equiv0$. Therefore the identity (\ref{3.2}) is equal to
\begin{eqnarray}
\Box f&=&-2\frac{|Z|^{2}}{(S+C)^{3}}
-2f^{2}+4\frac{{\rm Sm}({\rm Sic},{\rm Sic})}{S+C}\nonumber\\
&&- \ 2\alpha\left|\Delta\phi\frac{{\rm Sic}}{S+C}
-\nabla^{2}\phi\right|^{2}+2\alpha|\nabla^{2}\phi|^{2},\label{3.3}
\end{eqnarray}
where
\begin{equation}
f:=\frac{|{\rm Sic}|^{2}}{S+C}\label{3.4}
\end{equation}
which differs from the previous one in the proof of Theorem \ref{t2.2}. Integrating (\ref{3.3})
over $M$ yields
\begin{eqnarray}
\frac{d}{dt}\int_{M}f\!\ dV_{g(t)}
&=&\int_{M}\bigg[-2\frac{|Z|^{2}}{(S+C)^{3}}
-2f^{2}+4\frac{{\rm Sm}({\rm Sic},{\rm Sic})}{S+C}-fS\nonumber\\
&&- \ 2\alpha\left|\Delta\phi\frac{{\rm Sic}}{S+C}
-\nabla^{2}\phi\right|^{2}+2\alpha|\nabla^{2}\phi|^{2}\bigg]dV_{g(t)}.\label{3.5}
\end{eqnarray}
To control the integral of $|\nabla^{2}\phi|^{2}$ we recall the evolution equation for $|\nabla\phi
|^{2}$ (see \cite{Muller2012}, Proposition 4.3):
\begin{equation}
\Box|\nabla\phi|^{2}=-2\alpha|\nabla\phi\otimes\nabla\phi|^{2}
-2|\nabla^{2}\phi|^{2}.\label{3.6}
\end{equation}
In particular, we see that
\begin{equation}
|\nabla\phi|^{2}\leq\max_{M}\left(|\nabla\phi|^{2}\bigg|_{t=0}\right)=:A_{1}.\label{3.7}
\end{equation}
Moreover we have
\begin{equation*}
\frac{d}{dt}\int_{M}|\nabla\phi|^{2}dV_{g(t)}
\leq\int_{M}\left[-(S+C)|\nabla\phi|^{2}
-2|\nabla^{2}\phi|^{2}
+C|\nabla\phi|^{2}\right]dV_{g(t)}
\end{equation*}
which shows that
\begin{equation}
2\int^{t}_{0}\int_{M}|\nabla^{2}\phi|^{2}dV_{g(s)}ds+
\int_{M}|\nabla\phi|^{2}dV_{g(t)}\leq e^{Ct}A_{1}{\rm Vol}(M,g(0)).\label{3.8}
\end{equation}
Define
\begin{equation*}
A_{2}(t):=\int_{M}|\nabla^{2}\phi|^{2}dV_{g(t)}.
\end{equation*}
Plugging (\ref{3.8}) into (\ref{3.5}) we arrive at
\begin{equation}
\frac{d}{dt}\int_{M}f\!\ dV_{g(t)}
\leq\int_{M}\left[-2f^{2}+4\frac{{\rm Sm}({\rm Sic},{\rm Sic})}{S+C}
-fS\right]dV_{g(t)}+2\alpha A_{2}(t).\label{3.9}
\end{equation}

In the following we restrict ourself in $4D$ RHF, i.e., $m=\dim M=4$. In the case, the famous Gauss-Bonnet-Chern formula says that
\begin{equation}
2^{5}\pi^{2}\chi(M)
=\int_{M}\left[|{\rm Rm}|^{2}-4|{\rm Ric}|^{2}+R^{2}\right]dV_{g(t)}\label{3.10}
\end{equation}
for any $t\in[0,T]$. In order to use the formula (\ref{3.10}) we should translate the integrand
in (\ref{3.10}) into a function in terms of $S_{ijk\ell}$.

\begin{lemma}\label{l3.1} For any $m$-dimensional manifold $M$, one has
\begin{eqnarray}
|{\rm Rm}|^{2}-4|{\rm Ric}|^{2}+R^{2}&=&|{\rm Sm}|^{2}-4|{\rm Sic}|^{2}+S^{2}\nonumber\\
&&- \ \frac{m+9}{2}\alpha^{2}|\nabla\phi|^{4}
-9\alpha{\rm Sic}(\nabla\phi,\nabla\phi)+2\alpha S|\nabla\phi|^{2}.\label{3.11}
\end{eqnarray}
\end{lemma}

\begin{proof} Using $S_{ijk\ell}
=R_{ijk\ell}-\frac{\alpha}{2}(g_{j\ell}\nabla_{i}\phi\nabla_{k}
\phi+g_{k\ell}\nabla_{i}\phi\nabla_{j}\phi)$ we obtain
\begin{eqnarray*}
|{\rm Rm}|^{2}&=&R_{ijk\ell}R^{ijk\ell}\\
&=&\left(S_{ijk\ell}+\frac{\alpha}{2}(g_{j\ell}\nabla_{i}\phi\nabla_{k}
\phi+g_{k\ell}\nabla_{i}\phi\nabla_{j}\phi)\right)\\
&&\cdot \ \left(S^{ijk\ell}+\frac{\alpha}{2}(g^{j\ell}\nabla^{i}\phi\nabla^{k}
\phi+g^{k\ell}\nabla^{i}\phi\nabla^{j}\phi)\right)\\
&=&|{\rm Sm}|^{2}+\frac{m+1}{2}\alpha^{2}|\nabla\phi|^{4}
+\alpha\left(S_{ijk\ell}g^{j\ell}\nabla^{i}\phi\nabla^{k}\phi
+S_{ijk\ell}g^{k\ell}\nabla^{i}\phi\nabla^{j}\phi\right).
\end{eqnarray*}
Compute
\begin{eqnarray*}
S_{ijk\ell}g^{j\ell}&=&-R_{ik}-\frac{m+1}{2}\alpha\nabla_{i}
\phi\nabla_{k}\phi \ \ = \ \ -S_{ik}-\frac{m+3}{2}\alpha\nabla_{i}
\phi\nabla_{k}\phi,\\
S_{ijk\ell}g^{k\ell}&=&g^{k\ell}R_{ijk\ell}
-\frac{m+1}{2}\alpha\nabla_{i}\phi\nabla_{j}\phi \ \ = \ \ -\frac{m+1}{2}
\alpha\nabla_{i}\phi\nabla_{j}\phi
\end{eqnarray*}
because of the first Bianchi identity $g^{k\ell}R_{ijk\ell}=-g^{k\ell}(R_{jki\ell}+R_{kij\ell})
=-(-R_{ji}+R_{ij})=0$. Consequently
\begin{eqnarray*}
|{\rm Rm}|^{2}&=&|{\rm Sm}|^{2}+\frac{m+1}{2}\alpha^{2}|\nabla\phi|^{4}\\
&&+ \ \alpha\left[-{\rm Sic}(\nabla\phi,\nabla\phi)
-\frac{m+3}{2}\alpha|\nabla\phi|^{4}-\frac{m+1}{2}\alpha|\nabla\phi|^{4}\right]\\
&=&|{\rm Sm}|^{2}-\alpha{\rm Sic}(\nabla\phi,\nabla\phi)-\frac{m+3}{2}\alpha^{2}|\nabla\phi|^{4}.
\end{eqnarray*}
Similarly, we can show that
\begin{eqnarray*}
|{\rm Ric}|^{2}&=&|{\rm Sic}|^{2}+2\alpha{\rm Sic}(\nabla\phi,\nabla\phi)
+\alpha^{2}|\nabla\phi|^{4},\\
R^{2}&=&S^{2}+\alpha^{2}|\nabla\phi|^{4}+2\alpha S|\nabla\phi|^{2}.
\end{eqnarray*}
Combining those identities we obtain (\ref{3.11}).
\end{proof}

From (\ref{3.10}) and (\ref{3.11}) one has, in dimension $m=4$,
\begin{eqnarray}
\int_{M}\left[|{\rm Sm}|^{2}-4|{\rm Sic}|^{2}+S^{2}\right]dV_{g(t)}
&=&2^{5}\pi^{2}\chi(M)+\frac{13}{2}\alpha^{2}\int_{M}|\nabla\phi|^{4}dV_{g(t)}\nonumber\\
&&+ \ 9\alpha\int_{M}{\rm Sic}(\nabla\phi,\nabla\phi)dV_{g(t)}\label{3.12}\\
&&- \ 2\alpha\int_{M}S|\nabla\phi|^{2}dV_{g(t)}.\nonumber
\end{eqnarray}
Using the inequality
\begin{equation*}
{\rm Sic}(\nabla\phi,\nabla\phi)
\leq\epsilon|{\rm Sic}|^{2}+\frac{|\nabla\phi|^{4}}{4\epsilon}, \ \ \
\epsilon:=\frac{9}{26\alpha}
\end{equation*}
we obtain from (\ref{3.12}) that, in dimension $m=4$,
\begin{eqnarray}
\int_{M}\left[|{\rm Sm}|^{2}-4|{\rm Sic}|^{2}+S^{2}\right]dV_{g(t)}
&\leq&2^{5}\pi^{2}\chi(M)+13\alpha^{2}\int_{M}|\nabla\phi|^{4}dV_{g(t)}\nonumber\\
&&+ \ \frac{81}{26}\int_{M}|{\rm Sic}|^{2}dV_{g(t)}\label{3.13}\\
&&- \ 2\alpha\int_{M}S|\nabla\phi|^{2}dV_{g(t)}.\nonumber
\end{eqnarray}

For any $\epsilon>0$ we have
\begin{equation*}
4\frac{{\rm Sm}({\rm Sic}, {\rm Sic})}{S+C}\leq
\epsilon^{2}|{\rm Sm}|^{2}+\frac{4|{\rm Sic}|^{2}}{\epsilon^{2}
(S+C)^{2}}=\frac{4}{\epsilon^{2}}f^{2}+\epsilon^{2}|{\rm Sm}|^{2}
\end{equation*}
so that
\begin{eqnarray*}
&&-2f^{2}+4\frac{{\rm Sm}({\rm Sic},{\rm Sic})}{S+C}-fS \ \ \leq \ \ -2f^{2}+\frac{4}{\epsilon^{2}}f^{2}+\epsilon^{2}|{\rm Sm}|^{2}
-fS\\
&=&-2f^{2}+\epsilon^{2}\left(|{\rm Sm}|^{2}-4|{\rm Sic}|^{2}+S^{2}\right)
+4\epsilon^{2}|{\rm Sic}|^{2}-\epsilon^{2}S^{2}+\frac{4}{\epsilon^{2}}f^{2}-fs\\
&=&\epsilon^{2}\left(|{\rm Sm}|^{2}-4|{\rm Sic}|^{2}+S^{2}\right)
+(4\epsilon^{2}-1)fS-\left(2-\frac{4}{\epsilon^{2}}\right)f^{2}
+4C\epsilon^{2}f-\epsilon^{2}S^{2}.
\end{eqnarray*}
Using the estimate (\ref{3.13}) we have
\begin{eqnarray*}
&&\int_{M}\left[-2f^{2}+4\frac{{\rm Sm}({\rm Sic},{\rm Sic})}{S+C}
-fS\right]dV_{g(t)} \ \ \leq \ \ \epsilon^{2}\bigg[32\pi^{2}\chi(M)\\
&&+ \ 13\alpha^{2}A^{2}_{1}{\rm Vol}(M,g(0))e^{Ct}+\frac{81}{26}\int_{M}f(S+C)\!\ dV_{g(t)}\bigg]-\epsilon^{2}\int_{M}f^{2}\!\ dV_{g(t)}\\
&&+ \ (4\epsilon^{2}-1)\int_{M}fS\!\ dV_{g(t)}
-\left(2-\frac{4}{\epsilon^{2}}\right)\int_{M}f^{2}\!\ dV_{g(t)}
+4C\epsilon^{2}\int_{M}f\!\ dV_{g(t)}\\
&=&\epsilon^{2}\left[32\pi^{2}\chi(M)+13\alpha^{2}A^{2}_{1}{\rm Vol}(M,g(0))e^{Ct}\right]\\
&&+ \ \int_{M}\left[-\left(2-\frac{4}{\epsilon^{2}}\right)f^{2}
+\left(\frac{55}{26}+4\epsilon^{2}\right)fS+\left(\frac{81}{26}+4\epsilon^{2}\right)
Cf-\epsilon^{2}S^{2}\right]dV_{g(t)}.
\end{eqnarray*}
For any $\eta>0$ we have $fS\leq\eta f^{2}+\frac{1}{4\eta}S^{2}$ and then
\begin{eqnarray*}
&&\int_{M}\left[-2f^{2}+4\frac{{\rm Sm}({\rm Sic},{\rm Sic})}{S+C}
-fS\right]dV_{g(t)} \ \ \leq \ \ \epsilon^{2}\bigg[32\pi^{2}\chi(M)\\
&&+ \ 13\alpha^{2}A^{2}_{1}{\rm Vol}(M,g(0))e^{Ct}\bigg]
+\int_{M}\bigg[-\left(2-\frac{4}{\epsilon^{2}}
-\left(\frac{55}{26}+4\epsilon^{2}\right)\eta\right)f^{2}\\
&&+ \ \left(\frac{81}{26}+4\epsilon^{2}\right)Cf+
\left(\frac{\frac{55}{26}+4\epsilon^{2}}{4\eta}-\epsilon^{2}\right)S^{2}\bigg]dV_{g(t)}.
\end{eqnarray*}
If we choose
\begin{equation}
\eta=\frac{\frac{4}{\epsilon^{2}}}{\frac{55}{26}+4\epsilon^{2}}, \ \ \ \epsilon>2,\label{3.14}
\end{equation}
then
\begin{eqnarray*}
&&\int_{M}\left[-2f^{2}+4\frac{{\rm Sm}({\rm Sic},{\rm Sic})}{S+C}
-fS\right]dV_{g(t)} \ \ \leq \ \ \epsilon^{2}\bigg[32\pi^{2}\chi(M)\\
&&+ \ 13\alpha^{2}A^{2}_{1}{\rm Vol}(M,g(0))e^{Ct}\bigg]+\int_{M}\bigg[
-\left(2-\frac{8}{\epsilon^{2}}\right)
f^{2}+\left(\frac{81}{26}+4\epsilon^{2}\right)Cf\\
&&+ \ \left(\frac{(\frac{55}{26}+4\epsilon^{2})^{2}}{16}-1\right)\epsilon^{2}S^{2}\bigg]dV_{g(t)}.
\end{eqnarray*}
In particular, when $\epsilon=2\sqrt{2}$ in (\ref{3.14}) we arrive at
\begin{eqnarray}
&&\int_{M}\left[-2f^{2}+4\frac{{\rm Sm}({\rm Sic},{\rm Sic})}{S+C}
-fS\right]dV_{g(t)}\nonumber\\
&\leq&\int_{M}\left(-f^{2}+36Cf+574S^{2}\right)dV_{g(t)}\label{3.15}\\
&&+ \ 8\left[32\pi^{2}\chi(M)+13\alpha^{2}A^{2}_{1}{\rm Vol}(M,g(0))e^{Ct}\right].\nonumber
\end{eqnarray}
Plugging (\ref{3.15}) into (\ref{3.9}) implies
\begin{eqnarray}
\frac{d}{dt}\int_{M}f\!\ dV_{g(t)}&\leq&\int_{M}\left(-f^{2}+36Cf+574S^{2}\right)dV_{g(t)}\nonumber\\
&&+ \ \left[256\pi^{2}\chi(M)+104\alpha^{2}A^{2}_{1}{\rm Vol}(M,g(0))
e^{Ct}+2\alpha A_{2}(t)\right].\label{3.16}
\end{eqnarray}

\begin{theorem}\label{t3.2} Let $(M, g(t),\phi(t))_{t\in[0,T)}$ be a solution to RHF on a closed manifold $M$
with $m=\dim M=4$, $T\leq+\infty$, $\alpha(t)\equiv\alpha$ a positive constant. Choose a uniform constant
$C$ in
Theorem \ref{t2.2} such that $S_{g(t)}+C>0$. Then
\begin{eqnarray}
&&\int_{M}\frac{|{\rm Sic}_{g(s)}|^{2}_{g(s)}}{S_{g(s)}+C}dV_{g(s)}
+\int^{s}_{0}\int_{M}\frac{|{\rm Sic}_{g(t)}|^{4}_{g(t)}}{(S_{g(t)}+C)^{2}}dV_{g(t)}dt
\label{3.17}\\
&\leq&c_{0}(M, g(0), \phi(0), s)+574e^{36Cs}\int^{s}_{0}\int_{M}S^{2}_{g(t)}\!\ dV_{g(t)}dt,\nonumber
\end{eqnarray}
Furthermore,
\begin{eqnarray}
\int_{M}|{\rm Sic}_{g(s)}|_{g(s)}dV_{g(s)}&\leq&2c_{0}(M,g(0),\phi(0),s)
+\frac{C}{2}{\rm Vol}(M,g(s))\label{3.18}\\
&&+ \ 1148e^{36Cs}\int^{s}_{0}\int_{M}S^{2}_{g(t)}\!\ dV_{g(t)}dt,\nonumber\\
\int^{s}_{0}\int_{M}|{\rm Sic}_{g(t)}|^{2}_{g(t)}dV_{g(t)}dt&\leq&8 c_{0}(M, g(0),\phi(0),s)
+\frac{C^{2}}{4}\int^{s}_{0}{\rm Vol}(M, g(t))\!\ dt\nonumber\\
&&+ \ 4592 e^{36Cs}\int^{s}_{0}\int_{M}S^{2}_{g(t)}\!\ dV_{g(t)}dt,\label{3.19}\\
\int^{s}_{0}\int_{M}|{\rm Sm}_{g(t)}|^{2}_{g(t)}dV_{g(t)}dt&\leq&\frac{131}{50}C^{2}\int^{s}_{0}
{\rm Vol}(M, g(t))\!\ dt\nonumber\\
&&+ \ \frac{13\alpha^{2}A^{2}_{1}{\rm Vol}(M,g(0))}{C}\left(e^{Cs}-1\right)\nonumber\\
&&+ \ \frac{881}{25}c_{0}(M, g(0),\phi(0),s)+32\pi^{2}\chi(M)s\label{3.20}\\
&&+ \ \left(\frac{505694}{25}e^{36Cs}+\frac{81}{50}\right)\int^{s}_{0}
\int_{M}S^{2}_{g(t)}\!\ dV_{g(t)}dt,\nonumber
\end{eqnarray}
for all $s\in[0,T)$. Here, $A_{1}=\max_{M}|\nabla_{g(0)}\phi(0)|^{2}_{g(0)}$ and
\begin{eqnarray}
&&c_{0}(M, g(0),\phi(0),s)\nonumber\\
&=&\frac{256\pi^{2}\chi(M)}{36C}\left(e^{36Cs}-1\right)
+\frac{104\alpha^{2}A^{2}_{1}{\rm Vol}(M,g(0))}{35C}\left(e^{35Cs}
-e^{Cs}\right)\\
&&+ \ e^{37Cs}\alpha A_{1}{\rm Vol}(M, g(0))+e^{36Cs}\int_{M}\frac{|{\rm Sic}_{g(0)}|^{2}_{g(0)}}{S_{g(0)}
+C}dV_{g(0)}.\nonumber
\end{eqnarray}
Note that when $C<0$, the constant $c_{0}(M, g(0),\phi(0),s)$ can be chosen to be independent of $s$.
\end{theorem}

\begin{proof} By (\ref{3.16}) one has
\begin{eqnarray*}
\frac{d}{dt}\left(e^{-36Ct}\int_{M}f\!\ dV_{g(t)}\right)
+e^{-36Ct}\int_{M}f^{2}\!\ dV_{g(t)}&\leq&
574 e^{-36Ct}\int_{M}S^{2}\!\ dV_{g(t)}\\
&&+ \ e^{-36Ct}A_{3}
(t)
\end{eqnarray*}
where $A_{3}(t)=256\pi^{2}\chi(M)+104\alpha^{2}A^{2}_{1}{\rm Vol}(M,g(0))e^{Ct}
+2\alpha A_{2}(t)$. Therefore
\begin{eqnarray*}
&&e^{-36Cs}\int_{M}f\!\ dV_{g(s)}
+e^{-36Cs}\int^{s}_{0}\int_{M}f^{2}\!\ dV_{g(t)}dt\\
&\leq&\int^{s}_{0}A_{3}(t)e^{-36Ct}\!\ dt
+574\int^{s}_{0}e^{-36Ct}\int_{M}S^{2}\!\ dV_{g(t)}dt+\int_{M}f\!\ dV_{g(0)}\\
&\leq&\int^{s}_{0}\left[256\pi^{2}\chi(M)
+104\alpha^{2}A^{2}_{1}{\rm Vol}(M,g(0))e^{Ct}
+2\alpha\int_{M}|\nabla^{2}\phi|^{2}dV_{g(t)}\right]e^{-36Ct}dt\\
&&+ \ 574\int^{s}_{0}\int_{M}S^{2}\!\ dV_{g(t)}dt+\int_{M}f\!\ dV_{g(0)}\\
&=&\frac{256\pi^{2}\chi(M)}{36C}\left(1-e^{-36Cs}\right)
+\frac{104\alpha^{2}A^{2}_{1}{\rm Vol}(M,g(0))}{35C}\left(1-e^{-35Cs}\right)\\
&&+ \ 2\alpha\int^{s}_{0}e^{-36Ct}\int_{M}|\nabla^{2}\phi|^{2}dV_{g(t)}dt
+574\int^{s}_{0}\int_{M}S^{2}\!\ dV_{g(t)}dt+\int_{M}f\!\ dV_{g(0)}.
\end{eqnarray*}
According to the estimate (\ref{3.8}), we obtain the first inequality.

For the second statement, we use the following inequalities
\begin{equation*}
|{\rm Sic}|\leq\frac{|{\rm Sic}|^{2}}{S+C}+\frac{S+C}{4}, \ \ \
|S|\leq 2|{\rm Sic}|
\end{equation*}
so that
\begin{equation*}
|{\rm Sic}|\leq\frac{|{\rm Sic}|^{2}}{S+C}+\frac{|{\rm Sic}|}{2}+\frac{C}{4}
\end{equation*}
and
\begin{equation}
|{\rm Sic}|\leq2\frac{|{\rm Sic}|^{2}}{S+C}+\frac{C}{2}.\label{3.22}
\end{equation}
From (\ref{3.22}) and (\ref{3.17}) we obtain (\ref{3.18}).

For (\ref{3.19}), we observe that
\begin{eqnarray*}
|{\rm Sic}|^{2}&\leq&4\frac{|{\rm Sic}|^{4}}{(S+C)^{2}}+\frac{(S+C)^{2}}{16} \ \
\leq \ \ 4\frac{|{\rm Sic}|^{4}}{(S+C)^{2}}+\frac{S^{2}+C^{2}}{8}\\
&\leq&4\frac{|{\rm Sic}|^{4}}{(S+C)^{2}}+\frac{|{\rm Sic}|^{2}}{2}+\frac{C^{2}}{8}
\end{eqnarray*}
so that
\begin{equation*}
|{\rm Sic}|^{2}\leq 8\frac{|{\rm Sic}|^{4}}{(S+C)^{2}}+\frac{C^{2}}{4}.
\end{equation*}

For the last one, we use the inequality (\ref{3.13}) to deduce that
\begin{eqnarray*}
\int_{M}|{\rm Sm}|^{2}dV_{g(t)}&\leq&32\pi^{2}\chi(M)
+13\alpha^{2}A^{2}_{1}{\rm Vol}(M,g(0))e^{Ct}\\
&&+ \ \frac{81}{25}\int_{M}
\left(f^{2}+\frac{S^{2}+C^{2}}{2}\right)dV_{g(t)}+4\int_{M}|{\rm Sic}|^{2}dV_{g(t)}.
\end{eqnarray*}
Now the estimate (\ref{3.20}) follows from (\ref{3.17}) and (\ref{3.19}).
\end{proof}

\begin{theorem}\label{t3.3} Let $(M, g(t), \phi(t))_{t\in[0,T)}$ be a solution to RHF on a
closed manifold $M$ with $m=\dim M=4$, $T\leq+\infty$, $\alpha(t)\equiv\alpha$ a positive constant.
Suppose $\min_{M}S_{g(0)}>0$. Then
\begin{eqnarray}
\int_{M}|{\rm Sic}_{g(s)}|_{g(s)}dV_{g(s)}&\leq&2a_{0}(M,g(0),\phi(0),s)+1148\int^{s}_{0}\int_{M}
S^{2}_{g(t)}\!\ dV_{g(t)}dt,\label{3.23}\\
\int^{s}_{0}\int_{M}|{\rm Sic}_{g(t)}|^{2}_{g(t)}
dV_{g(t)}dt&\leq&8a_{0}(M,g(0),\phi(0),s)\nonumber\\
&&+ \ 4592\int^{s}_{0}\int_{M}S^{2}_{g(t)}dV_{g(t)}dt,\label{3.24}\\
\int^{s}_{0}\int_{M}|{\rm Sm}_{g(t)}|^{2}_{g(t)}\!\ dV_{g(t)}dt&\leq&32\pi^{2}\chi(M)s
+13(\alpha A_{1})^{2}{\rm Vol}(M,g(0))s\nonumber\\
&&+ \ \frac{881}{25}a_{0}(M, g(0),\phi(0),s)\label{3.25}\\
&&+ \ \frac{1011469}{50}\int^{s}_{0}\int_{M}S^{2}_{g(t)}\!\ dV_{g(t)}dt\nonumber
\end{eqnarray}
for all $s\in[0,T)$. Here
\begin{eqnarray}
a_{0}(M, g(0),\phi(0),s)&:=&256\pi^{2}\chi(M)s+104(\alpha A_{1})^{2}
{\rm Vol}(M,g(0))s\nonumber\\
&&+ \ \alpha A_{1}{\rm Vol}(M,g(0))+\int_{M}\frac{|{\rm Sic}_{g(0)}|^{2}_{g(0)}}{S_{g(0)}}dV_{g(0)}.
\label{3.26}
\end{eqnarray}
\end{theorem}

\begin{proof} We use the fact $|{\rm Sic}|\leq2|{\rm Sic}|^{2}/S$ and Theorem \ref{t3.2} (where
we take $C=0$). Then
\begin{equation*}
\int_{M}|{\rm Sic}|dV_{g(s)}\leq2a_{0}(M,g(0),\phi(0),s)+1148\int^{s}_{0}\int_{M}S^{2}\!\
dV_{g(t)}dt.
\end{equation*}
The inequality (\ref{3.24}) follows from Theorem \ref{t3.2} and $|{\rm Sic}|^{2}
\leq 8|{\rm Sic}|^{4}/S^{2}$. The last estimate follows from the inequality mentioned at the end
of the proof of Theorem \ref{t3.2}.
\end{proof}

According to Theorem \ref{t2.3} and following \cite{Simon2015}, we consider the basic assumption
({\bf BA}) for a solution $(M, g(t),\phi(t))_{t\in[0,T)}$ to RHF:
\begin{itemize}

\item[(a)] $M$ is a connected and closed $4$-dimensional smooth manifold,

\item[(b)] $(M, g(t),\phi(t))_{t\in[0,T)}$ is a solution to RHF with $\alpha(t)\equiv\alpha$ a positive constant,

\item[(c)] $T<\infty$,

\item[(d)] $\max_{M\times[0,T)}|S_{g(t)}|\leq1$.

\end{itemize}
The upper bound $1$ in condition (d) is not essential, since we can rescale the pair
$(g(t), \phi(t))$ so that the condition (d) is always satisfied. Furthermore, since
\begin{equation*}
|\nabla_{g(t)}
\phi(t)|^{2}_{g(t)}\leq A_{1}
\end{equation*}
(by (\ref{3.6})) it follows that condition (d) is equivalent to the uniform bound for $R_{g(t)}$.

\begin{theorem}\label{t3.4} If $(M, g(t),\phi(t))_{t\in[0,T)}$ satisfies {\bf BA}, then
\begin{eqnarray}
\frac{d}{dt}\int_{M}f\!\ dV_{g(t)}&\leq&\int_{M}
\left(-f^{2}+88f\right)dV_{g(t)}\nonumber\\
&&+ \ \left[128\pi^{2}\chi(M)
+52(\alpha A_{1})^{2}{\rm Vol}(M,g(0))e^{2t}
+2\alpha A_{2}(t)\right].\label{3.27}
\end{eqnarray}
where $f:=|{\rm Sic}_{g(t)}|^{2}_{g(t)}/(S_{g(t)}+2)$.
\end{theorem}

\begin{proof} Recall the estimate
\begin{eqnarray*}
4\frac{{\rm Sm}({\rm Sic},{\rm Sic})}{S+2}
&\leq&\frac{4|{\rm Sic}|^{4}}{\epsilon^{2}(S+2)^{2}}+
\epsilon^{2}|{\rm Sm}|^{2}\\
&\leq&\frac{4}{\epsilon^{2}}f^{2}
+\epsilon^{2}\left(|{\rm Sm}|^{2}
-4|{\rm Sic}|^{2}+S^{2}\right)+4\epsilon^{2}|{\rm Sic}|^{2}\\
&\leq&\frac{4}{\epsilon^{2}}f^{2}+\epsilon^{2}\left(|{\rm Sm}|^{2}
-4|{\rm Sic}|^{2}+S^{2}\right)
+4\epsilon^{2}(S+2)f.
\end{eqnarray*}
Since $-1\leq S\leq 1$, it follows that $4\epsilon^{2}(S+2)f\leq12\epsilon^{2}f$. Hence
\begin{equation*}
4\frac{{\rm Sm}({\rm Sic},{\rm Sic})}{S+2}
\leq\frac{4}{\epsilon^{2}}f^{2}+\epsilon^{2}
\left(|{\rm Sm}|^{2}-4|{\rm Sic}|^{2}+S^{2}\right)
+12\epsilon^{2}f.
\end{equation*}
Using the inequality $-fS=-f(s+2)+2f\leq 2f$ and (\ref{3.13}), we arrive at
\begin{eqnarray*}
&&\int_{M}\left[-2f^{2}+4\frac{{\rm Sm}({\rm Sic},{\rm Sic})}{S+2}
-fS\right]dV_{g(t)}\\
&\leq&\int_{M}\left[-2f^{2}+\frac{4}{\epsilon^{2}}f^{2}
+\epsilon^{2}\left(|{\rm Sm}|^{2}
-4|{\rm Sic}|^{2}+S^{2}\right)
+12\epsilon^{2}f-fS\right]dV_{g(t)}\\
&=&\int_{M}\left[-\left(2-\frac{4}{\epsilon^{2}}\right)
f^{2}+(12\epsilon^{2}+2)f+\epsilon^{2}\left(|{\rm Sm}|^{2}
-4|{\rm Sic}|^{2}+S^{2}\right)\right]dV_{g(t)}\\
&\leq&\int_{M}\left[-\left(2-\frac{4}{\epsilon^{2}}\right)f^{2}
+(12\epsilon^{2}+2)f\right]dV_{g(t)}\\
&&+ \ \epsilon^{2}\left[32\pi^{2}\chi(M)
+13\epsilon^{2}(\alpha A_{1})^{2}
{\rm Vol}(M,g(0))e^{2t}\right]+\frac{243}{26}\epsilon^{2}
\int_{M}f\!\ dV_{g(t)}\\
&=&\int_{M}\left[-\left(2-\frac{4}{\epsilon^{2}}\right)
f^{2}+\left(\frac{555}{26}\epsilon^{2}+2\right)f\right]dV_{g(t)}\\
&&+ \ \epsilon^{2}
\left[32\pi^{2}\chi(M)
+13\epsilon^{2}(\alpha A_{1})^{2}{\rm Vol}(M,g(0))e^{2t}\right].
\end{eqnarray*}
From (\ref{3.9}) we get (\ref{3.27}) (compare with (\ref{3.16}) when $C=2$).
\end{proof}

\begin{theorem}\label{t3.5} If $(M, g(t),\phi(t))_{t\in[0,T)}$ satisfies {\bf BA}, then
\begin{eqnarray}
\int_{M}|{\rm Sic}_{g(s)}|^{2}_{g(s)}dV_{g(s)}&\leq&b(M,g(0),\phi(0),s),\label{3.28}\\
\int_{M}|{\rm Sm}_{g(s)}|^{2}_{g(s)}dV_{g(s)}&\leq&32\pi^{2}\chi(M)
+13(\alpha A_{1})^{2}{\rm Vol}(M,g(0))e^{2s}\nonumber\\
&&+ \ \frac{185}{26}b(M,g(0),\phi(0),s),\label{3.29}\\
\int^{s}_{0}\int_{M}|{\rm Sic}_{g(t)}|^{4}_{g(t)}dV_{g(t)}dt&\leq&b(M,g(0),\phi(0),s),\label{3.30}\\
\int^{T}_{s}\int_{M}|{\rm Sic}|^{p}_{g(t)}dV_{g(t)}dt&\leq&\left[|b(M,g(0),\phi(0),T)|\right]^{
\frac{p}{4}}
e^{\frac{T(4-p)}{4}}\nonumber\\
&&\left[{\rm Vol}(M,g(0))\right]^{\frac{4-p}{4}}(T-s)^{\frac{4-p}{4}}\label{3.31}
\end{eqnarray}
for any $s\in[0,T)$ and $0<p<4$. Here
\begin{eqnarray}
b(M,g(0),\phi(0),s)&:=&9e^{88s}\int_{M}|{\rm Sic}_{g(0)}|^{2}_{g(0)}dV_{g(0)}
+\frac{1152}{88}\pi^{2}\chi(M)\left(e^{88s}-1\right)\nonumber\\
&&+ \ \frac{468}{86}(\alpha A_{1})^{2}{\rm Vol}(M,g(0))\left(e^{88s}
-e^{2s}\right)\label{3.32}\\
&&+ \ 9(\alpha A_{1}){\rm Vol}(M,g(0))e^{90s}.\nonumber
\end{eqnarray}
\end{theorem}

\begin{proof} Write $A_{3}(t):=128\pi^{2}\chi(M)
+52(\alpha A_{1})^{2}{\rm Vol}(M,g(0))e^{2t}+2\alpha A_{2}(t)$. The inequality (\ref{3.27}) implies
\begin{equation*}
\frac{d}{dt}\int_{M}f\!\ dV_{g(t)}\leq A_{3}(t)+\int_{M}\left(-f^{2}+88f\right)dV_{g(t)}
\end{equation*}
and then
\begin{equation*}
\frac{d}{dt}\left(e^{-88t}\int_{M}f\!\ dV_{g(t)}\right)
\leq-e^{-88t}\int_{M}f^{2}\!\ dV_{g(t)}
+e^{-88t}A_{3}(t).
\end{equation*}
Therefore
\begin{eqnarray*}
&&e^{-88s}\int^{s}_{0}\int_{M}f^{2}\!\ dV_{g(t)}dt
+e^{-88s}\int_{M}f\!\ dV_{g(s)} \ \ \leq \ \
\int_{M}f\!\ dV_{g(0)}\\
&&+ \ \int^{s}_{0}e^{-88t}\bigg[128\pi^{2}\chi(M)
+52(\alpha A_{1})^{2}{\rm Vol}(M,g(0))e^{2t}
+2\alpha A_{2}(t)\bigg]dt\\
&=&\int_{M}f\!\ dV_{g(0)}
+\frac{128}{88}\pi^{2}\chi(M)\left(1-e^{-88s}\right)
+\frac{52}{86}
(\alpha A_{1})^{2}{\rm Vol}(M,g(0))\left(1-e^{-86s}\right)\\
&&+ \ 2\alpha\int^{s}_{0}\int_{M}|\nabla^{2}\phi|^{2}\!\ dV_{g(t)}dt\\
&\leq&\int_{M}f\!\ dV_{g(0)}+\frac{128}{88}\pi^{2}\chi(M)
\left(1-e^{-88s}\right)
+\frac{52}{86}(\alpha A_{1})^{2}{\rm Vol}(M,g(0))
\left(1-e^{-86s}\right)\\
&&+ \ (\alpha A_{1}){\rm Vol}(M,g(0))e^{2s}
\end{eqnarray*}
by (\ref{3.8}). Because $|S|\leq1$, we have $\frac{1}{3}|{\rm Sic}|^{2}\leq f\leq|{\rm Sic}|^{2}$ and hence
\begin{eqnarray*}
&&\frac{e^{-88s}}{9}\int^{s}_{0}\int_{M}|{\rm Sic}|^{4}\!\ dV_{g(t)}dt
+\frac{e^{-88s}}{3}\int_{M}|{\rm Sic}|^{2}\!\ dV_{g(s)}\\
&\leq&\int_{M}|{\rm Sic}|^{2}\!\ dV_{g(0)}+\frac{128}{88}\pi^{2}\chi(M)
\left(1-e^{-88s}\right)\\
&&+ \ \frac{52}{86}(\alpha A_{1})^{2}{\rm Vol}(M,g(0))
\left(1-e^{-86s}\right)+(\alpha A_{1}){\rm Vol}(M,g(0))e^{2s}.
\end{eqnarray*}
This estimate yields (\ref{3.28}) and (\ref{3.30}).

For (\ref{3.29}), we use (\ref{3.13}):
\begin{eqnarray*}
\int_{M}|{\rm Sm}|^{2}dV_{g(t)}
&\leq&32\pi^{2}\chi(M)+13(\alpha A_{1})^{2}{\rm Vol}(M,g(0))e^{2t}\\
&&+ \ \frac{81}{26}\int_{M}|{\rm Sic}|^{2}dV_{g(t)}+4\int_{M}|{\rm Sic}|^{2}dV_{g(t)}\\
&=&32\pi^{2}\chi(M)+13(\alpha A_{1})^{2}{\rm Vol}(M,g(0))e^{2t}
+\frac{185}{26}b(M,g(0),\phi(0),t)
\end{eqnarray*}
by (\ref{3.28}).

For (\ref{3.31}), we use
\begin{equation*}
\frac{d}{dt}{\rm Vol}(M,g(t))
=-\int_{M}S\!\ dV_{g(t)} \ \ \ \text{and} \ \ \ -1\leq S\leq 1
\end{equation*}
to deduce
\begin{equation*}
e^{-T}{\rm Vol}(M,g(0))\leq{\rm Vol}(M,g(t))\leq e^{T}{\rm Vol}(M,g(0)).
\end{equation*}
Consequently, for any $0<s<r<T$,
\begin{eqnarray*}
&&\int^{r}_{s}\int_{M}|{\rm Sic}|^{p}dV_{g(t)}dt \ \ \leq \ \ \left(\int^{r}_{s}\int_{M}|{\rm Sic}|^{4}dV_{g(t)}dt\right)^{\frac{p}{4}}
\left(\int^{r}_{s}\int_{M}dV_{g(t)}dt\right)^{\frac{4-p}{4}}\\
&\leq& \left[|b(M,g(0),\phi(0),T)|\right]^{\frac{p}{4}}
(r-s)^{\frac{4-p}{4}}e^{\frac{T(4-p)}{4}}
\left[{\rm Vol}(M,g(0))\right]^{\frac{4-p}{4}}.
\end{eqnarray*}
Thus we get (\ref{3.31}).
\end{proof}

Define
\begin{eqnarray}
c(M,g(0),\phi(0),T)&:=&9e^{90T}\bigg[\int_{M}|{\rm Sic}_{g(0)}|^{2}_{g(0)}dV_{g(0)}
+\pi^{2}|\chi(M)|\nonumber\\
&&+ \ [(\alpha A_{1})^{2}+(\alpha A_{1})]{\rm Vol}(M,g(0))\bigg].\label{3.33}
\end{eqnarray}
Then $|b(M,g(0),\phi(0),T)|\leq c(M,g(0),\phi(0),T)$. Theorem \ref{t3.5} now yields
\begin{eqnarray}
\sup_{t\in[0,T)}\int_{M}|{\rm Sic}_{g(t)}|^{2}_{g(t)}dV_{g(t)}
&\leq&c(M,g(0),\phi(0),T) \ \ < \ \ +\infty,\label{3.34}\\
\sup_{t\in[0,T)}\int_{M}|{\rm Sm}_{g(t)}|^{2}_{g(t)}dV_{g(t)}
&\leq&32\pi^{2}\chi(M)+8c(M,g(0),\phi(0),T)\label{3.35}\\
&&+ \ 13(\alpha A_{1})^{2}{\rm Vol}(M,g(0))e^{2T} \ \ < \ \ +\infty.\nonumber
\end{eqnarray}

\section{4D Ricci-harmonic flow with bounded $S_{g(t)}$: II}\label{section4}

We in this section consider the general case of $4D$ RHF $(M,g(t),\phi(t))_{t\in[0,T)}$
with $T\leq\infty$ and $\alpha(t)\geq0, \dot{\alpha}\leq0$. Using the same notions as (\ref{3.1}) and (\ref{3.4}) we have
\begin{eqnarray}
\Box f&=&-2\frac{|Z|^{2}}{(S+C)^{3}}
-2f^{2}+4\frac{{\rm Sm}({\rm Sic},{\rm Sic})}{S+C}-2\alpha\left|\Delta\phi\frac{{\rm Sic}}{S+C}
-\nabla^{2}\phi\right|^{2}\nonumber\\
&&+ \ 2\alpha|\nabla^{2}\phi|^{2}+\frac{\dot{\alpha}}{(S+C)^{2}}
|\nabla\phi|^{2}|{\rm Sic}|^{2}
+2\dot{\alpha}(S+C)
\left\langle{\rm Sic},\nabla\phi\otimes\nabla\phi\right\rangle.\label{4.1}
\end{eqnarray}
Using $f=|{\rm Sic}|^{2}/(S+C)$ we get
\begin{eqnarray*}
&&\frac{\dot{\alpha}}{(S+C)^{2}}
|\nabla\phi|^{2}|{\rm Sic}|^{2}
+2\dot{\alpha}(S+C)
\left\langle{\rm Sic},\nabla\phi\otimes\nabla\phi\right\rangle\\
&=&\frac{\dot{\alpha}}{S+C}|\nabla\phi|^{2}f-2\dot{\alpha}(S+C)
\left[\epsilon f(S+C)+\frac{1}{4\epsilon}|\nabla\phi|^{4}\right]\\
&=&-\dot{\alpha}f\left[2\epsilon(S+C)^{2}-\frac{|\nabla\phi|^{2}}{S+C}\right]
-\frac{\dot{\alpha}}{2\epsilon}(S+C)|\nabla\phi|^{4}\\
&\leq&-\dot{\alpha}(S+C)^{4}|\nabla\phi|^{2}.
\end{eqnarray*}
where $\epsilon:=|\nabla\phi|^{2}/2(S+C)^{3}$. Hence, together with (\ref{4.1}) yields
\begin{eqnarray}
\Box f&\leq&-2\frac{|Z|^{2}}{(S+C)^{3}}
-2f^{2}+4\frac{{\rm Sm}({\rm Sic},{\rm Sic})}{S+C}\nonumber\\
&&+ \ 2\alpha|\nabla^{2}\phi|^{2}-\dot{\alpha}(S+C)^{4}|\nabla\phi|^{2}.\label{4.2}
\end{eqnarray}
As in the proof of (\ref{3.9}) we arrive at
\begin{eqnarray}
\frac{d}{dt}\int_{M}f\!\ dV_{g(t)}
&\leq&\int_{M}\left[-2f^{2}+4\frac{{\rm Sm}({\rm Sic},{\rm Sic})}{S+C}
-fS\right]dV_{g(t)}\nonumber\\
&&+ \ 2\alpha A_{2}(t)-\dot{\alpha}A_{1}\int_{M}(S+C)^{4}dV_{g(t)}.\label{4.3}
\end{eqnarray}
Using Lemma \ref{l3.1} yields
\begin{eqnarray}
\frac{d}{dt}\int_{M}f\!\ dV_{g(t)}&\leq&\int_{M}\left(-f^{2}+36C f
+574 S^{2}\right)dV_{g(t)}-\dot{\alpha}(t)A_{1}\int_{M}(S+C)^{4}dV_{g(t)}\nonumber\\
&&+ \ \left[256\pi^{2}\chi(M)
+104\alpha^{2}(t)A^{2}_{1}{\rm Vol}(M,g(0))e^{Ct}+2\alpha(t) A_{2}(t)\right].\label{4.4}
\end{eqnarray}
The above estimate can help use to extend Theorem \ref{t3.2} and Theorem \ref{t3.3} to the general RHF. For 
example, we have
\begin{eqnarray}
&&\int_{M}\frac{|{\rm Sic}_{g(s)}|^{2}_{g(s)}}{S_{g(s)}+C}dV_{g(s)}
+\int^{s}_{0}\int_{M}\frac{|{\rm Sic}_{g(t)}|^{4}_{g(t)}}{(S_{g(t)}+C)^{2}}dV_{g(t)}dt\nonumber\\
&\leq&c_{0}(M, g(0),\phi(0),s)
+574 e^{36Cs}\int^{s}_{0}\int_{M}S^{2}_{g(t)}dV_{g(t)}dt\label{4.5}\\
&&+ \ e^{36Cs}\int^{s}_{0}-\dot{\alpha}(t)\int_{M}\left(S_{g(t)}+C\right)^{4}dV_{g(t)}
dt.\nonumber
\end{eqnarray}
Similarly, we can extend Theorem \ref{t3.5} to the general RHF by the same argument if an 
analogous result of Theorem \ref{t3.4} holds for the general RHF. However, this is obvious, because alone 
the outline of the proof of Theorem \ref{t3.4} we have
\begin{eqnarray}
\frac{d}{dt}\int_{M}f\!\ dV_{g(t)}
&\leq&\int_{M}\left(-f^{2}+88f\right)dV_{g(t)}-\dot{\alpha}(t)A_{1}\int_{M}(S+C)^{4}dV_{g(t)}\nonumber\\
&&+ \ \left[128\pi^{2}\chi(M)
+52(\alpha(t) A_{1})^{2}{\rm Vol}(M, g(0))e^{2t}
+2\alpha(t) A_{2}(t)\right]\label{4.6}
\end{eqnarray}
if $(M, g(t), \phi(t))_{t\in[0,T)}$ satisfies {\bf BA}.
\\

{\bf Acknowledgments.} The main results were announced in the 2015 Chinese-German workshop on metric
Riemannian geometry at Shanghai Jiao Tong University from Oct. 12 to Oct. 16. The author thanks
the organizers of this workshop.




\begin{thebibliography}{99}

\bibitem{Abolarinwa2015} Abolarinwa, A. \textit{Entropy formulas and their applications on time dependent Riemannian metrics}, Electron. J. Math. Anal. Appl., {\bf 3}(2015), no. 1, 77--88. MR3280632

\bibitem{Bailesteanu2013a} Bailesteanu, Mihai. \textit{On the heat kernel under the Ricci flow
coupled with the harmonic map flow}, arXiv: 1309.0138.

\bibitem{Bailesteanu2013b} Bailesteanu, Mihai. \textit{Gradient estimates for the heat equation under the Ricci-harmonic map flow}, arXiv: 1309.0139.

\bibitem{Bailesteanu-Tran2013} Bailesteanu, Mihai; Tranh, Hung. \textit{Heat kernel estimates under the Ricci0harmonic map flow}, arXiv: 1310.1619.

\bibitem{Bamler-Zhang2015} Bamler, Richard H.; Zhang, Qi S. \textit{Heat kernel and curvature bounds in Ricci flowa with bounded scalar curvature}, arXiv: 1501.01291.

\bibitem{Cao2011} Cao, Xiaodong. \textit{Curvature pinching estimate and singularities of the Ricci
flow}, Comm. Anal. Geom., {\bf 19}(2011), no. 5, 975--990. MR2886714


\bibitem{Cao-Guo-Tran2015} Cao, Xiaoding; Guo, hongxin; Tran, Hung. \textit{Harnack estimates for conjugate heat kernel on evolving manifolds}, Math. Z., {\bf 281}(2015), no. 1-2, 201--214. MR3384867.

\bibitem{Chen-Wang2013} Chen, Xiuxiong; Wang, Bing. \textit{On the conditions to extend Ricci flow (III)}, Int. Math. Res. Not. IMRN 2013, no. 10, 2349--2367. MR3061942.

\bibitem{CZ2013} Cheng, Liang; Zhu, Anqiang. \textit{On the extension of the harmonic Ricci flow}, Geom. Dedicata, {\bf 164}(2013), 179--185. MR3054623

\bibitem{CLN2006} Chow, Bennett; Lu, Peng; Ni, Lei. \textit{Hamilton's Ricci flow}, Gradient Studies
in Mathematics, {\bf 77}, American Mathematical Society, Providence, RI; Science Press, New York, 2006. xxxvi+608 pp. ISBN: 978-0-8281-4231-7; 0-8218-4231-5 (MT2274812) (2008a: 53068)


\bibitem{Enders-Muller-Topping2011} Enders, Joerg; Muller, Reto; Topping, Peter M. \textit{On type-I
singularities in Ricci flow}, Comm. Anal. Geom., {\bf 19}(2011), no. 5, 905--922. MR2886712.


\bibitem{Fang2013a} Fang, Shouwen. \textit{Differential Harnack inequalities for heat equations with potentials under geoemtric flows}, Arch. Math. (Basel), {\bf 100}(2013), no. 2, 179--189. MR3020132.

\bibitem{Fang2013b} Fang, Shouwen. \textit{Differential Harnack estimates for backward heat equations
with potentials under an extended Ricci flow}, Adv. Geom., {\bf 13}(2013), no. 4, 741--755. MR3181545.



\bibitem{Fang-Zheng2015} Fang, Shouwen; Zheng, Tao. \textit{An upper bound of the heat kernel
along the harmonic-Ricci flow}, arXiv: 1501.00639.

\bibitem{Fang-Zheng2016} Fang, Shouwen; Zheng, Tao. \textit{The (logarithmic) Sobolev
inequalities alonf geometric flow and applications}, J. Math. Anal. Appl., {\bf 434}(2016), no. 1,
729--764. MR3404584.


\bibitem{Fang-Zhu2015} Fang, Shouwen; Zhu, Peng. \textit{Differential Harnack estimates for backward heat equations with potentials under geometric flows}, Commun. Pure Appl. Anal., {\bf 14}(2015), no. 3, 793--809. MR3320151.



\bibitem{Guo-He2014} Guo, Hongxin; He, Tongtong. \textit{Harnack estimates for geometric flow,
applications to Ricci flow coupled with harmonic map flow}, Geom. Dedicata, {\bf 169}(2014), 411--418.
MR3175257.


\bibitem{Guo-Huang-Phong2015} Guo, Bin; Huang, Zhijie; Phong, Duong H. \textit{Pseudo-locality for
a coupled Ricci flow}, arXiv: 1510.04332.


\bibitem{Guo-Ishida2014} Guo, Hongxin; Ishida, Masashi. \textit{Harnack estimates for nonlinear
backward heat equations in geometric flows}, J. Funct. Anal., {\bf 267}(2014), no. 8, 2638--2662.
MR3255470.

\bibitem{Guo-Philipowski-Thalmaier2013} Guo, Hongxin; Philipowski, Robert; Thalmairt; Anton. \textit{Entropy and lowest eigenvalue on evolving manifolds}, Pacific J. Math., {\bf 264}(2013), no. 1, 61--81. MR3079761.

\bibitem{Guo-Philipowski-Thalmaier2014} Guo, Hongxin; Philipowski, Robert; Thalmairt; Anton. \textit{A stochastic approach to the harmonic map heat flow on manifolds with time-dependent Riemannian metric}, Stochastic Process. Appl., {\bf 124}(2014), no. 11, 3535--3552. MR3249346.


\bibitem{Guo-Philipowski-Thalmaier2015} Guo, Hongxin; Philipowski, Robert; Thalmairt; Anton. \textit{An entropy formula for the heat equation on manifolds with time-dependent metric, application to ancient solutions}, Potential
    Anal., {\bf 42}(2015), no. 2, 483--497. MR3306693.

\bibitem{H1982} Hamilton, Richard S. \textit{Three-manifolds with positive Ricci curvature},
J. Differential Geom., {\bf 17}(1982), no. 2, 255--306. MR0664497 (84a: 53050)

\bibitem{Li2014} Li, Yi. \textit{Eigenvalues and entropies under the harmonic-Ricci flow}, Pacific J. Math., {\bf 267}(2014), no. 1, 141--184. MR3163480


\bibitem{List2005} List, Bernhard. \textit{Evolution of an extended Ricci flow system}, PhD thesis, AEI Potsdam, 2005.


\bibitem{List2008} List, Bernhard. \textit{Evolution of an extended Ricci flow system}, Comm.
Anal. Geom., {\bf 16}(2008), no. 5, 1007--1048. MR2471366 (2010i: 53126)

\bibitem{Liu-Wang2014} Liu, Xian-Gao; Wang, Kui. \textit{A Gaussian upper bound of the conjugate heat equation along an extended Ricci flow}, arXiv: 1412.3200.

\bibitem{Lott-Sesum2014} Lott, J.; Sesum, N. \textit{Ricci flow on three-dimensional manifolds with
symmetry}, Comment. Math. Helv., {\bf 89}(2014), no. 1, 1--32


\bibitem{Ma-Cheng2010} Ma, Li; Cheng, Liang. \textit{On the conditions to control curvature tensors
of Ricci flow}, Ann. Global Anal. Geom., {\bf 37}(2010), no. 4, 403--411. MR2601499 (2011d: 53158)

\bibitem{Muller2009} M\"uller, Reto. \textit{The Ricci flow coupled with harmonic map flow}, PhD thesis, ETH Z\"urich, doi: 10.3929/ethz-a-005842361, 1009.

\bibitem{Muller2010} M\"uller, Reto. \textit{Monotone volume formulas for geoemtric flow}, J. Reine
Angew. Math., {\bf 643}(2010), 39--57. MR2658189 (2011k: 53086)

\bibitem{Muller2012} M\"uller, Reto. \textit{Ricci flow coupled with harmonic map flow}, Ann.
Sci. \'Ec. Norm. Sup\'er. (4) {\bf 45}(2012), no. 1, 101--142. MR2961788


\bibitem{Muller-Rupflin2015} M\"uller, Reto; Rupflin, Melanie.
\textit{Smooth long-time existence of harmonic Ricci flow on surfaces},
arXiv: 1510.03643.


\bibitem{P1} Perelman, Grisha. \textit{The entropy formula for the Ricci flow and its geoemtric
applications}, arXiv: math/0211159.

\bibitem{P2} Perelman, Grisha. \textit{Ricci flow with surgery on three-manifolds}, arXiv: math/0303109.

\bibitem{P3} Perelman, Grisha. \textit{Finite extincition time for the solutions to the
Ricci flow on certain three-manifolds}, arXiv: math/0307245.

\bibitem{Sesum2005} Sesum, Natasa. \textit{Curvature tensor under the Ricci flow}, Amer.
J. Math., {\bf 127}(2005), no. 6, 1315--1324. MR2183526 (2006f:53097)

\bibitem{Simon2015} Simon, Miles. \textit{$4D$ Ricci flows with bounded scalar curvature},
arXiv: 1504.02623v1.


\bibitem{Tadano2015a} Tadano, Homare. \textit{A lower diameter bound for compact domain manifolds
of shrinking Ricci-harmonic solitons}, Kodai Math. J., {\bf 38}(2015), no. 2, 302--309. MR3368067.


\bibitem{Tadano2015b} Tadano, Homare. \textit{Gap theorems for Ricci-harmonic solitons}, arXiv: 1505.03194.


\bibitem{Tian-Zhang2013} Tian, Gang; Zhang, Zhenlei. \textit{Regularity
of K\"ahler-Ricci flows on Fano manifolds}, arXiv: 1310.5897.

\bibitem{Wang2008} Wang, Bing. \textit{On the conditions to extend Ricci flow}, Int. Math. res. Not.
IMRN 2008, no. 8, Art. ID rnn012, 30pp. MR2428146 (2009k: 53176)

\bibitem{Wang2012} Wang, Bing. \textit{On the conditions to extend Ricci flow (II)}, Int. Math. Res. Not. IMRN 2012, no. 14, 3192--3223. MR2946223.

\bibitem{WLF2012} Wang, Lin Feng. \textit{Differential Harnack inequalities under a coupled Ricci flow}, Math. Phys.
Anal. Geom., {\bf 15}(2012), no. 4, 343--360. MR2996456.

\bibitem{Williams2015} Williams, M. B. \textit{Results on coupled Ricci and harmonic map flows}, Adv.
Geom., {\bf 15}(2015),
no. 1, 7--25. MR3300708.

\bibitem{Ye2008a} Ye, Rugang. \textit{Curvature estimates for the Ricci flow. I}, Calc. Var. Partial Differential Equations, {\bf 31}(2008), no. 4, 417--437. MR2372898 (2009f: 53106)

\bibitem{Ye2008} Ye, Rugang. \textit{Curvature estimates for the Ricci flow. II}, Calc. Var. Partial Differential Equations, {\bf 31}(2008), no. 4, 439--466. MR2372899 (2009f: 53107)

\bibitem{Zhang2010} Zhang, Zhou. \textit{Scalar curvature behavior for finite-time singularity of
K\"ahler-Ricci flow}, Michigan Math., {\bf 59}(2010), no. 2, 419--433. MR2677630 (2011j: 53128)

\bibitem{Zhu2013} Zhu, Anqiang. \textit{Differential Harnack inequalities for the backward heat equation with
potential under the harmonic-Ricci flow}, J. Math. Anal. Appl., {\bf 406}(2013), no. 2, 502--510. MR3062555.

\end{thebibliography}
\end{document}